\documentclass[a4paper,12pt]{amsart}

\usepackage{amsmath}
\usepackage{amssymb}
\usepackage{amsfonts}

\newtheorem{thm}{Theorem}[section]
\newtheorem{prop}[thm]{Proposition}
\newtheorem{lem}[thm]{Lemma}
\newtheorem{cor}[thm]{Corollary}
\newtheorem{rem}[thm]{Remark}
\newtheorem{rems}[thm]{Remarks}
\newtheorem{defi}[thm]{Definition}
\newtheorem{exo}{\bf\large Exercice}

\newcommand{\R}{\mathbb{R}}

\newcommand{\N}{\mathbb{N}}
\newcommand{\C}{\mathbb{C}}
\newcommand{\ds}{\displaystyle}

\newcommand{\I}{\infty}

\newcommand{\Sum}{\displaystyle \sum}

\newcommand{\Int}{\displaystyle \int}

\newcommand{\Inf}{\displaystyle \inf}

\newcommand{\Lim}{\displaystyle \lim}

\newcommand{\beq}{\begin{eqnarray}}
\newcommand{\eeq}{\end{eqnarray}}
\newcommand{\bq}{\begin{equation}}
\newcommand{\eq}{\end{equation}}
\newcommand{\beqn}{\begin{eqnarray*}}
\newcommand{\eeqn}{\end{eqnarray*}}
\newcommand{\bex}{\begin{exo}}
\newcommand{\eex}{\end{exo}}
\newcommand{\ben}{\begin{enumerate}}
\newcommand{\een}{\end{enumerate}}

\let\wt=\widetilde

\def\refeq#1{~(\ref{#1})}

\def\inte#1{
\displaystyle\mathop{#1\kern0pt}^\circ }

\def\cC{{\mathcal C}}
\def\cD{{\mathcal D}}

\def\cL{{\mathcal L}}

\def\cP{{\mathcal P}}

\def\cS{{\mathcal S}}

\def\na{\nabla}

\newcommand{\EQ}[1]{\begin{equation} \begin{split} #1
 \end{split} \end{equation}}

\let\wt=\widetilde

\newcommand{\lec}{{\, \lesssim \, }}

\date{\today}

\begin{document}

  \title[Scattering for the critical 2-D NLS with exponential growth]{Scattering for the critical 2-D NLS with exponential growth}

\author[H. Bahouri]{Hajer Bahouri}
\address[H. Bahouri]%
{Laboratoire d'Analyse et de Math{\'e}matiques Appliqu{\'e}es UMR 8050 \\
Universit\'e Paris-Est Cr\'eteil \\
61, avenue du G{\'e}n{\'e}ral de Gaulle\\
94010 Cr{\'e}teil Cedex, France}
\email{hbahouri@math.cnrs.fr}

\author[S. Ibrahim]{Slim Ibrahim}
\address[S. Ibrahim]%
{Department of Mathematics and Statistics\\
University of Victoria\\
PO Box 3060, STN CSC,\\
Victoria, BC, V8P 5C3, Canada}
\email{ibrahims@uvic.ca}
\thanks{The second author  is partially supported by NSERC\# 371637-2009 grant.}

\author[G. Perelman]{Galina Perelman}
\address[G. Perelman]%
{Laboratoire d'Analyse et de Math{\'e}matiques Appliqu{\'e}es UMR 8050 \\
Universit\'e Paris-Est Cr\'eteil \\
61, avenue du G{\'e}n{\'e}ral de Gaulle\\
94010 Cr{\'e}teil Cedex, France}
\email{galina.perelman@u-pec.fr}

 \begin{abstract}
 In  this article, we establish in the radial framework the $H^1$-scattering for  the critical 2-D nonlinear Schr\"odinger equation with exponential growth. Our strategy relies on both the a priori estimate derived in \cite{CGT, PV} and the characterization  of the lack of
compactness of the Sobolev embedding of $H_{rad}^1(\R^2)$ into the critical Orlicz space ${\cL}(\R^2)$ settled in \cite{BMM}.    The radial setting, and particularly  the fact that we
deal with bounded functions far away from the origin, occurs  in a
crucial way in our approach.  \end{abstract}
\subjclass[2010]{}

 \maketitle



\maketitle \tableofcontents


\section{Introduction and statement of   the results}
\subsection{Setting of the problem and main result}
We are interested in the two dimensional   nonlinear Schr\"odinger equation:
\begin{equation}
\label{NLS} \left\{
 \begin{aligned}
  &i\partial_t u+\Delta u = f(u), \\
  &u_{|t=0}=u_{0}\in H_{rad}^1(\R^2),
 \end{aligned}
\right.
\end{equation}
 where the function $u$  with complex values depends on $(t,x) \in \R \times \R^2$, and  the  nonlinearity $f:\C\to\C$ is defined by
\begin{equation} \label{def f}
 f(u) = \left({\rm e}^{4\pi |u|^2}-1-4\pi |u|^2\right)u.
\end{equation}
Let us emphasize that  the solutions
  of  the Cauchy problem \eqref{NLS}-\eqref{def f} formally satisfy the conservation of mass
and Hamiltonian
\begin{equation}
\label{mass} M(u,t):=\int_{\R^2}|u(t,x)|^2  dx \quad \mbox{and}
\end{equation}
\begin{equation}
\label{hamil} H(u,t):=\int_{\R^2}\, \Big(|\na u (t,x)|^2 +F(u(t,x))\Big)\, dx,
\end{equation}
where
$$
F(u)=\frac{1}{4\pi}\left({\rm e}^{4\pi |u|^2}-1-4\pi |u|^2-8\pi^2 |u|^4 \right).
$$

\smallskip
 The question of the existence of global solutions for the Cauchy problem \eqref{NLS}-\eqref{def f} was investigated in  \cite{CIMM} and  subcritical,
critical and supercritical regimes in the energy space was
identified.  This notion of criticality is related to the size  of the initial
Hamiltonian $H(u_0)$ with respect to
$1$. More precisely,   the concerned Cauchy problem  is said to be subcritical if $H(u_0) < 1$,  critical if $H(u_0) = 1$ and  supercritical  if $H(u_0) > 1$. \\

\smallskip

 In \cite{CIMM},  the authors established   in both subcritical and critical regimes  the existence of global solutions in the functional space ${\cC}(\R, H^1(\R^2))\cap L_{loc}^4(\R, W^{1,4}(\R^2)) $, and proved that
well-posedness fails to hold in the supercritical one. Thereafter in \cite{IMMN2}, the scattering problem for the concerned nonlinear Schr\"odinger equation  has been solved in the subcritical case.\\

\smallskip
Note that several works have been devoted to the     nonlinear Schr\"odinger equation  \eqref{NLS}. In particular,  one can mention the following  a priori estimate proved independently by Colliander-Grillakis-Tzirakis and by Planchon-Vega in
\cite{CGT, PV}
 \begin{equation}
\label{estpriori} \|u\|_{L^4(\R,L^8)} \lec \|u\|_{L^\I(\R,L^2)}^{3/4} \|\na u\|_{L^\I(\R,L^2)}^{1/4}, \end{equation}
available for any global solution $u$  in $L^\I(\R,H^1)$. \\

\bigskip

\bigskip

  The purpose  of this paper is  to investigate the critical case  $H(u_0)=1$ in the radial framework, and to establish   that the $H^1$-scattering also holds in that case. More precisely, our   main  result states as follows:

\begin{thm}
 \label{Main-NLS}
Let  $u$ be a solution to \eqref{NLS}-\eqref{def f}  satisfying
$H(u)=1$, then
\begin{equation} \label{Stg}
u\in L^4(\R, W^{1,4}(\R^2)),
\end{equation}
where \begin{equation}\label{defS'h}
W^{1,4}(\R^2) :=  \left\{f \in {\mathcal S}' (\R^2), \:   \| f\|_{L^4} + \| \nabla f\|_{L^4} < \infty \right\} \, .
\end{equation}
Besides there exist $v_{\pm} \in H_{rad}^1(\R^2)$ such that
$$
\|u (t, \cdot) - {\rm e}^{ i t \Delta} \,v_{\pm}\|_{H^1}\stackrel{t\to\pm\infty}\longrightarrow 0.
$$  \end{thm}


   \subsection{General scheme of the proof}\label{schgen}
   All along this article,  we shall see that the norm $L^\infty(\R, L^4(\R^2))$  will play a decisive  role in the approach adopted to establish our  result.   The main  difficulty  lies in the non-conservation of the $L^4$-norm over time for solutions to the free Schr\"odinger equation.  To investigate the behavior  of this norm,  we shall   resort to the a priori estimate provided in \cite{CGT, PV} taking advantage
of  the fact that   in view of the radial setting, $H^1_{rad}(\R^2)$ embeds on the one hand compactly into
$L^4(\R^2)$ and on the other hand  in $L^\infty(\R^2\setminus\{0\})$.  It turns out that  thanks to this   a priori estimate at hand, the proof of our result  does not require structure theorems as those obtained in \cite{BG, ker, MV}  which played a crucial role for instance in the remarkable work of  \cite{km}.\\

   Roughly speaking, the proof of our main result is done in three steps. In the  first step, we establish that  for any solution $u$ of the Cauchy problem \eqref{NLS}-\eqref{def f}  and any positive real sequence $(t_n)_{n \geq 0}$  tending to $+ \infty$,  the evolution of $u(t_n,\cdot)$ under the flow of  the linear Schr\"odinger equation  converges to zero  in $L^\infty(\R_+, L^4)$. This step constitutes the heart of the matter and the key ingredient to achieve it  is the a priori estimate derived in \cite{CGT, PV}. In   the  second step, we  highlight a lack of compactness at infinity making use of  the virial identity.  Finally, in the third step we complete the proof of our main theorem  by distinguishing two cases: a first case where the norm in $L^\infty(\R_+, \wt {\mathcal L})$ is strictly less than
$ \frac{1}{\sqrt{4\pi}}$ and that we will qualify by the case where the whole mass does not concentrate, and a second case where the norm in $L^\infty(\R_+, \wt {\mathcal L})$ is equal to  $ \frac{1}{\sqrt{4\pi}}$ and that we will designate  by the case where the whole mass  concentrates. The main idea to handle the second  case  which is the more challenging is the explicit description of that situation by means of the example by Moser settled in \cite{BMM, BMM1}. To understand that case, we undertake an analysis  depending on whether the whole mass  concentrates in  small or large times.

\subsection{Layout of the paper}

The  paper is organized as follows:   in Section \ref{Tt} we provide the basic tools which are used in this text, namely critical 2-D Sobolev embeddings,  an overview of the lack of compactness of $H^1(\R^2)$ into the Orlicz space and basic facts about the  linear Schr\"odinger equation. In Section \ref{Vid}, we establish several useful estimates. This includes virial identity and properties of solutions to the nonlinear Schr\"odinger equation associated to Cauchy data evolving sub-critically under the flow of the linear equation. Section \ref{sec:mainth}  is devoted to the proof of our main result.  As it is mentioned in Paragraph \ref{schgen}, this is achieved  in three steps:  a first step where the strong convergence to zero of the sequence $({\rm e}^{ i t \Delta}\, u(t_n,\cdot))$ in $L^\infty(\R_+, L^4)$ is settled for any real sequence $(t_n)_{n \geq 0}$  tending to $+ \infty$, a second step where a lack of compactness at infinity is emphasized making use of  virial identity, and lastly a third step where the proof is complete.
Finally, we deal in appendix with
various Moser-Trudinger type inequalities which are of constant use all along this article.

 \medskip

 Finally, we mention that the letter~$C$ will be used to denote a universal constant
which may vary from line to line. We also use~$A\lesssim B$   to
denote an estimate of the form~$A\leq C B$  for some
constant~$C$.
 For simplicity, we shall also still denote by $(u_n)$ any
subsequence of $(u_n)$ and designate by $\circ(1)$ any sequence which tends to $ 0 $ as $n$ goes to infinity.

\section{Technical tools}\label{Tt}


\subsection{Critical 2-D Sobolev embedding}
It is well known that $H^1(\R^2)$   embeds continuously into $L^p(\R^2)$ for all $2\leq
p<\infty$ but not in $L^\infty(\R^2)$.
However, resorting to an
interpolation argument, we can estimate the $L^\I$ norm of
functions in $H^1({\mathbb R}^2)$, using a stronger norm but with
a weaker growth (namely logarithmic). More precisely, we have the
following logarithmic estimate which will be needed in this paper:
\begin{lem}[\cite{DlogSob}, Theorem 1.3]
\label{Hmu}
 Let $0<\alpha<1$.  For any $\lambda>\frac{1}{2\pi\alpha}$ and
any $0<\mu\leq1$, a constant $C_{\lambda}>0$ exists such that for
any function $u\in H^1(\R^2)\cap{\mathcal C}^\alpha(\R^2)$, we
have
\begin{equation}
\label{H-mu} \|u\|^2_{L^\infty}\leq
\lambda\|u\|_{H_\mu}^2\log\left(C_{\lambda} +
\frac{8^\alpha\mu^{-\alpha}\|u\|_{{\mathcal
C}^{\alpha}}}{\|u\|_{H_\mu}}\right),
\end{equation}
where ${ {\mathcal C}}^\alpha$ denotes the inhomogeneous H\"older
space of regularity index $\alpha$ and  $H_\mu$  the Sobolev space
endowed with the norm $\|u\|_{H_\mu}^2:=\|\nabla
u\|_{L^2}^2+\mu^2\|u\|_{L^2}^2$.
\end{lem}  Otherwise in the radial case which is the setting of this article, we have the following
estimate which implies the control of the $L^\infty$-norm far away
from the origin (see for instance \cite{BMM}):
\begin{lem}
\label{apa3}
 Let $u\in H^1_{rad}(\R^2)$ and $1\leq p<\infty$. Then
\bq \label{BoundLP}
|u(x)|\leq\frac{C_p}{r^{\frac{2}{2+p}}}\,\|u\|_{L^p}^{\frac{p}{p+2}}\|\nabla
u\|_{L^2}^{\frac{2}{p+2}},
 \eq
with $ r= |x|$.  In particular
\begin{eqnarray}
\label{Bound}
|u(x)|&\leq&\frac{C}{r^{\frac{1}{2}}}\,\|u\|_{L^2}^{\frac{1}{2}}\|\nabla
u\|_{L^2}^{\frac{1}{2}} \quad \mbox{and} \\ \label{redestl4bis} |u(x)|&\leq&
\frac{C}{r^{\frac{1}{3}}}\,\|u\|_{L^4}^{\frac{2}{3}}\|\nabla
u\|_{L^2}^{\frac{1}{3}}.\end{eqnarray}
\end{lem}
\begin{rem}
In the general case,  the embedding of $H^1 (\R^2)$ into $L^p
(\R^2)$ is not compact as it is shown for instance by the example: $u_n(x)=\varphi(x+x_n)$ with
$\varphi$ a function  belonging  to ${\mathcal D}(\R^2)$ and $(x_n)$ a sequence of $\R^2$ satisfying  $|x_n|\to\infty$. However,  in the radial setting,
the following compactness result holds (see for example \cite{BL, Kavian, Strauss}):
\end{rem}
\begin{lem}
\label{compacity} Let $2< p<\infty$. The embedding of $H^1_{rad}(\R^2)$ into
$L^p(\R^2)$ is compact.
\end{lem}
\noindent For our subject, it will be useful to point out  the following refined estimate:
\begin{lem}
\label{refest} There is a positive constant $C$ such that
\begin{equation}
\label{apa4}  \|u\|_{L^4}\leq C \|\nabla u\|^{\frac 1 4}_{L^2}\|u\|^{\frac 3 4}_{H^1},
\end{equation}
for any $u \in H^1 (\R^2)$.
\end{lem}
 \begin{proof}
In view of the continuity of the Fourier transform
$${\mathcal F} : L^{\frac 4 3 } (\R^2)\longrightarrow L^4(\R^2),$$
it suffices to prove that
$$ \|\widehat{u}\|_{L^{\frac 4 3 }(\R^2)}\leq C \|\nabla u\|^{\frac 1 4}_{L^2(\R^2)}\|u\|^{\frac 3 4}_{H^1(\R^2)},$$
where $\widehat{u}$ denotes the Fourier transform of $u$. \\

To go to this end, let us begin by observing that  \footnote{We used the classical notation  $\langle \xi \rangle = (1+|\xi|^2)^{\frac 1 2} $.}
\begin{eqnarray*}
 \int_{\R^2} \big|\widehat{u}(\xi)\big|^{\frac 4 3}  d\xi &=& \int_{\R^2} \big|\widehat{u}(\xi) |\xi|^{\frac 1 4} \langle \xi \rangle^{\frac 3 4} \big|^{\frac 4 3}  \big(|\xi|^{\frac 1 4} \langle \xi \rangle^{\frac 3 4}  \big)^{-\frac 4 3} d\xi   \\ &\leq & \Big(\int_{\R^2} \big|\widehat{u}(\xi)\big|^{2} |\xi|^{\frac 1 2} \langle \xi \rangle^{\frac 3 2}   \ d\xi \Big)^{\frac 2 3}\Big(\int_{\R^2} \frac {d\xi} { |\xi| \langle \xi \rangle^{3}} \Big)^{\frac 1 3}  \\ &\lesssim & \Big(\int_{\R^2} \big|\widehat{u}(\xi) |\xi|\big|^{\frac 1 2} \big|\widehat{u}(\xi) \langle \xi \rangle \big|^{\frac 3 2}    \ d\xi \Big)^{\frac 2 3}.\end{eqnarray*}
 Now applying H\"older inequality, we deduce that $$ \int_{\R^2} \big|\widehat{u}(\xi)\big|^{\frac 4 3}  d\xi \lesssim \Big(\int_{\R^2} |\xi|^{2}\big|\widehat{u}(\xi)\big|^{2}  d\xi \Big)^{\frac 1 6}\Big(\int_{\R^2}  \langle \xi \rangle^{2}\big|\widehat{u}(\xi)\big|^{2}  d\xi \Big)^{\frac 1 2},$$
 which thanks to Fourier-Plancherel formula gives rise  to
 $$ \|\widehat{u}\|_{L^{\frac 4 3 }(\R^2)}\lesssim \|\nabla u\|^{\frac 1 4}_{L^2(\R^2)}\|u\|^{\frac 3 4}_{H^1(\R^2)}.$$
 This ends the proof of the lemma.
 \end{proof}

  \medbreak

 Furthermore  \begin{equation}
\label{2Dembed}H^1(\R^2)\hookrightarrow {\mathcal L},\end{equation}
where ${\mathcal L}$ denotes the  Orlicz space $L^\phi$ associated to the function   $\phi= {\rm e}^{ s^2}-1$  (see Definition \ref{deforl} below).
This embedding stems immediately from the
following sharp Moser-Trudinger type inequalities (see \cite{AT, M, Ruf, Tru}):
\begin{prop}\label{proptm}
\begin{equation}
\label{Mos2} \sup_{\|u\|_{H^1}\leq 1}\;\;\int_{\R^2}\,\left({\rm
e}^{4\pi |u(x)|^2}-1\right)\,dx:=\kappa<\infty,
\end{equation}
\end{prop}
\noindent and states as follows
\begin{equation}
\label{2D-embed} \|u\|_{{\mathcal
L}}\leq\frac{1}{\sqrt{4\pi}}\|u\|_{H^1},\end{equation}
\noindent
when the Orlicz space  ${\mathcal L}$ is endowed with the
norm $\|\cdot\|_{\mathcal L}$ where the number $1$ in Definition \ref{deforl}  is replaced by
the above constant $\kappa$.\\

 \medskip

 In this article  we are rather interested in the  Sobolev embedding
 \begin{equation}
\label{lackor2}
H^1(\R^2)\hookrightarrow\wt{\mathcal L},
\end{equation}  where $\wt {\mathcal L}$ is the  Orlicz space $L^\phi$ associated to the function   $\phi= {\rm e}^{ s^2}-1 - s^2$, and
which arises naturally in the study of the nonlinear Schr\"odinger equation  with
exponential growth \eqref{NLS}-\eqref{def f}.  It is obvious that
\begin{equation}
\label{2D-embedbis} \|u\|_{\wt {\mathcal
L}}\leq\frac{1}{\sqrt{4\pi}}\|u\|_{H^1},\end{equation}
 where  $\|\cdot \|_{\wt {\mathcal
L}}$ is the Orlicz norm introduced in  Definition \ref{deforl}, with the constant $\kappa$ appearing  in Identity  \eqref{Mos2} instead of the number $1$. Besides, as it can be shown by the  example by Moser  $
f_{\alpha_n}$ given by  \eqref{Mos-LIons}, the Sobolev constant appearing in \eqref{2D-embedbis} is optimal.
\\

For our purpose,  we shall resort to  the following  Moser-Trudinger type inequalities  and the  resulting corollaries that  will be demonstrated  in Appendix \ref{usefullMT}:
\begin{prop}
\label{usconsbis} Let $\alpha\in [0,4\pi [$ and $p$ be a nonnegative real larger than $2$.
A constant $C(\alpha,p)$ exists
such that
\begin{equation}
\label{Mosbis2} \int_{\R^2}\,{\rm e}^{\alpha
|u(x)|^2}\,|u(x)|^p\, dx\leq C(\alpha,p) \Int_{\R^2}\,|u(x)|^p\, dx,
\end{equation}
for all $u$ in $H^1(\R^2)$ satisfying $\|\nabla
u\|_{L^2(\R^2)}\leq1$.
\end{prop}
A byproduct of  Proposition \ref{usconsbis} is the following useful result.
\begin{prop} \label{mostrud} Let $\alpha\in [0,4\pi [$. There is a  constant $c_\alpha$
such that
\begin{equation}
\label{Mosbis} \int_{\R^2}\,\left({\rm e}^{\alpha
|u(x)|^2}-1- \alpha
|u(x)|^2\right)\,dx\leq c_\alpha \|u\|_{L^4}^4,
\end{equation}
for all $u$ in $H^1(\R^2)$ satisfying $\|\nabla
u\|_{L^2(\R^2)}\leq1$.  \end{prop}
From \eqref{Mosbis}, it is easy to deduce the following consequence.
\begin{cor}
\label{cormoser} Let $(u_n)$ be a bounded sequence in
$H^1(\R^2)$ such that
$$
 \|u_n\|_{L^4}\stackrel{n\to\infty}\longrightarrow 0,
$$
then
\begin{equation}
\label{condL4}
\|u_n\|_{\wt {\mathcal
L}}\leq\frac{1}{\sqrt{4\pi}}\|\nabla u_n\|_{L^2}+\circ(1),\quad n\to\infty.
\end{equation}
\end{cor}
\begin{rem}
Let us point out that  estimate \eqref{2D-embed}  (respectively \eqref{2D-embedbis}) fails  if we replace in the right hand side $\|u\|_{H^1}$ by $\|\nabla u\|_{L^2}$.  To be convinced, just consider  the sequence $(u_n)_{n \geq  0}$ defined by
 $u_n(x):=\frac{1}{n}\,{\rm e}^{-|\frac{x}{n}|^2}$ (respectively  $u_n(x):=\frac{1}{\sqrt{n}}\,{\rm e}^{-|\frac{x}{n}|^2}$). \end{rem}
Inequality \eqref{Mosbis2} fails for $\alpha = 4\pi $ as it can be shown by the example by Moser defined by \eqref{Mos-LIons}. However, the following estimate needed in the sequel occurs:
\begin{cor}
\label{uscons*} For any  $\delta> 0$, there exit   $c_\delta$  and $\varepsilon_0$ such that for all $0 < \varepsilon \leq \varepsilon_0$ and all nonnegative  real $p \geq 2$, there is a positive  constant $C(\delta,\varepsilon,p)$ such that for  $\ds  r = \frac 1 {1- \varepsilon \, c_\delta}$ the following estimate holds
\begin{equation}
\label{Mosbis3} \int_{\R^2}\, {\rm e}^{4\pi (1+\varepsilon)
|u(x)|^2}\, |u(x)|^p dx\leq C(\delta,\varepsilon,p) \Big( \|u\|^p_{L^p(\R^2)} + \|u\|^{p}_{L^{pr}(\R^2)} \Big),
\end{equation}
 for all $u$ in $H^1(\R^2)$ satisfying $\|\nabla
u\|_{L^2(\R^2)}\leq1$ and   $\|u\|_{\wt {\mathcal
L}} \leq \frac 1 {\sqrt{4\pi (1+ 2 \delta)}}\cdot$
\end{cor}
Let us close this section by  introducing the definition of  the so-called Orlicz spaces on $\R^d$.
\begin{defi}\label{deforl}\quad\\
Let $\phi : \R^+\to\R^+$ be a convex increasing function such that
$$
\phi(0)=0=\lim_{s\to 0^+}\,\phi(s),\quad
\lim_{s\to\infty}\,\phi(s)=\infty.
$$
We say that a measurable function $u : \R^d\to\C$ belongs to
$L^\phi$ if there exists $\lambda>0$ such that
$$
\int_{\R^d}\,\phi\Big(\frac{|u(x)|}{\lambda}\Big)\,dx<\infty.
$$
We denote then \bq \label{norm}
\|u\|_{L^\phi}=\Inf\,\left\{\,\lambda>0,\quad\Int_{\R^d}\,\phi\Big(\frac{|u(x)|}{\lambda}\Big)\,dx\leq
1\,\right\}. \eq
\end{defi}

 \subsection{Development on  the lack of
compactness of Sobolev embedding  in the Orlicz space}
The Sobolev embeddings \eqref{2D-embed}
and \eqref{2D-embedbis} are non compact at
least for two reasons. The first reason is the lack of compactness
at infinity that we can highlight through the sequence $u_n(x)=\varphi(x+x_n)$ where
$0\neq\varphi\in{\mathcal D}$ and $|x_n|\to\infty$. The second
reason is of concentration-type derived by   J. Moser  in \cite{M} and by P.-L.  Lions
in \cite{Lions1, Lions2} and is illustrated by the following  fundamental
sequence  $(f_{\alpha_n})_{n\geq 0}$, when $(\alpha_n)_{n\geq 0}$ is a sequence of positive reals tending to infinity:
\begin{eqnarray}\label{Mos-LIons}
 f_{\alpha_n}(x)&=&\; \left\{
\begin{array}{cllll}\sqrt{\frac{\alpha_n}{2\pi}}\quad&\mbox{if}&\quad |x|\leq {\rm
e}^{-\alpha_n},\\\\ -\frac{\log|x|}{\sqrt{2\alpha_n\pi}} \quad
&\mbox{if}&\quad {\rm e}^{-\alpha_n}\leq |x|\leq 1 ,\\\\
0 \quad&\mbox{if}&\quad
|x|\geq 1.
\end{array}
\right.
\end{eqnarray}
Indeed, one can prove by straightforward  computations  (detailed for instance  in \cite{BMM}) that
$\|f_{\alpha_n}\|_{\mathcal L}\stackrel{n\to\infty}\longrightarrow  \frac{1}{\sqrt{4\pi}}$
and $\|f_{\alpha_n}\|_{\wt{\mathcal L}}\stackrel{n\to\infty}\longrightarrow  \frac{1}{\sqrt{4\pi}} \cdot$\\

\medskip
\noindent In \cite{BMM, BMM2, BMM1},  the lack of compactness of the critical Sobolev
embedding
$$ H^1(\R^2)\hookrightarrow {\mathcal L}(\R^2) \label{embedorlicz}
$$
  was  described in terms  of an asymptotic decomposition by means of generalization of the above
example by Moser. To  state  this characterization in a clear way, let us recall  some definitions.
\begin{defi}
\label{orthogen} We shall designate by  a scale any sequence $\underline{\alpha}:=(\alpha_n)$
of positive real numbers going to infinity and by a profile any  function $\psi$ belonging to the set $$
{\cP}:=\Big\{\;\psi\in L^2(\R,{\rm e}^{-2s}ds);\;\;\; \psi'\in
L^2(\R)\;\;\mbox{and}\;\;\psi_{|]-\infty,0]}=0\,\Big\}.
$$
Two scales $\underline{\alpha}$, $\underline{\beta}$  are said orthogonal (
in short $\underline{\alpha}\perp\underline{\beta}$) if
    $$
   \Big|\log\left({\beta_n}/{\alpha_n}\right)\Big|\stackrel{n\to\infty}\longrightarrow\infty.
    $$
\end{defi}
\medskip
\noindent The  asymptotically orthogonal  decomposition  derived in \cite{BMM} is formulated in the following terms:
\begin{thm}
\label{main} Let $(u_n)_{n \geq 0}$ be a bounded sequence in
$H^1_{rad}(\R^2)$ such that $$
u_n\rightharpoonup 0, \quad
\limsup_{n\to\infty}\|u_n\|_{\mathcal L}=A_0>0 \quad \quad
\mbox{and} $$  \bq \label{compl2}   \lim_{R\to\infty}\;
\limsup_{n\to\infty}\,\int_{|x|>R}\,|u_n(x)|^2\,dx=0.\eq Then, there
exist a sequence $(\underline{\alpha}^{(j)})$ of pairwise
orthogonal scales and a sequence of profiles $(\psi^{(j)})$ in
${\cP}$ such that, up to a subsequence extraction, we have for all
$\ell\geq 1$, \bq \label{decomp}
u_n(x)=\Sum_{j=1}^{\ell}\,\sqrt{\frac{\alpha_n^{(j)}}{2\pi}}\;\psi^{(j)}\left(\frac{-\log|x|}{\alpha_n^{(j)}}\right)+{\rm
r}_n^{(\ell)}(x),\quad\limsup_{n\to\infty}\;\|{\rm
r}_n^{(\ell)}\|_{\mathcal
L}\stackrel{\ell\to\infty}\longrightarrow 0. \eq Moreover, we have
the following stability estimates \bq \label{ortogonal} \|\nabla
u_n\|_{L^2(\R^2)}^2=\Sum_{j=1}^{\ell}\,\|{\psi^{(j)}}'\|_{L^2(\R)}^2+\|\nabla
{\rm r}_n^{(\ell)}\|_{L^2(\R^2)}^2+\circ(1),\quad n\to\infty. \eq
\end{thm}
\begin{rems}\label{remsprof}\quad \\
\begin{itemize}
\item The example  by Moser  can be written as $$
f_{\alpha_n}(x)= \sqrt{\frac{\alpha_n}{2\pi}}\;{\mathbf L}\Big(\frac{-\log|x|}{\alpha_n}\Big),
$$
where
\begin{eqnarray} \label{lionsprof}
{\mathbf L}(s)&=&\; \left\{
\begin{array}{cllll}0 \quad&\mbox{if}&\quad
s\leq 0,\\ s \quad
&\mbox{if}&\quad 0\leq s\leq 1,\\
1 \quad&\mbox{if}&s\geq 1.
\end{array}
\right.
\end{eqnarray}
\item Let us  emphasize that it was proved in \cite{BMM} that \begin{equation}
\label{OrliczMax} \|u_n\|_{\cL}\stackrel{n\to\infty}\longrightarrow  \sup_{j\geq
1}\,\left(\lim_{n\to\infty}\,\|g_n^{(j)}\|_{\cL}\right),
\end{equation}
where~$g_n^{(j)}(x) = \sqrt{\frac{\alpha_n^{(j)}}{2\pi}}\;\psi^{(j)}\left(\frac{-\log|x|}{\alpha_n^{(j)}}\right)$, and
\bq \label{profile}
\Lim_{n\to\infty}\,\|g_n^{(j)}\|_{{\cL}}=\frac{1}{\sqrt{4\pi}}\,\max_{s>0}\;\frac{|\psi^{(j)}(s)|}{\sqrt{s}}\cdot \eq
\item Therefore (see \cite{BMM1} for a detailed proof),  if in addition to  the assumptions of Theorem \ref{main} the sequence  $(u_n)_{n \geq 0}$ satisfies
$$\|u_n\|_{{\mathcal
L}}=  \frac{1}{\sqrt{4\pi}}\|\nabla u_n\|_{L^2} + \circ(1),$$ then we have necessary
\begin{equation} \label{formprofile}u_n(x)=\sqrt{\frac{\alpha_n}{2\pi}}\;{\mathbf L}\left(\frac{-\log|x|}{\alpha_n}\right)+{\rm
r}_n(x),\quad \|\nabla {\rm
r}_n\|_{L^2} \to 0 \quad \mbox{as} \quad n\to\infty, \end{equation} with ${\mathbf L}$ the Lions profile given by \eqref{lionsprof}.
\item Taking advantage of the above remark,  we infer that if a bounded sequence $(u_n)_{n \geq 0}$  in
$H^1_{rad}(\R^2)$ converges to zero in $L^4(\R^2)$ and satisfies
\begin{equation} \label{assumtild}\|u_n\|_{\wt{\mathcal
L}}=  \frac{1}{\sqrt{4\pi}}\|\nabla u_n\|_{L^2}+ \circ(1),\end{equation}
then \begin{equation} \label{formprofiletild}u_n(x)=\sqrt{\frac{\alpha_n}{2\pi}}\;{\mathbf L}\left(\frac{-\log|x|}{\alpha_n}\right)+{\rm
r}_n(x),\quad \|\nabla {\rm
r}_n\|_{L^2} \to 0 \quad \mbox{as} \quad n\to\infty. \end{equation}
Indeed, write
$$ u_n= \chi u_n + (1-\chi) u_n,$$
where $\chi$ is a radial function  in~$\mathcal{D}(\R^2)$ equal to one in the unit ball and valued in  $[0,1]$. Thus by virtue of the radial estimate \eqref{redestl4bis}, we have $$ \| \wt v_n\|_{L^\infty} \stackrel{n\to\infty}\longrightarrow 0,$$ where $\wt v_n:=(1-\chi)  u_n$,  which implies that
 $$  \| \wt v_n\|_{\wt{\mathcal
L}} \stackrel{n\to\infty}\longrightarrow 0.$$
 Therefore \eqref{assumtild} also reads
  \begin{equation} \label{bisest} \|v_n\|_{\wt{\mathcal
L}}=  \frac{1}{\sqrt{4\pi}}\|\nabla u_n\|_{L^2}+ \circ(1),\end{equation}
with  $v_n:= \chi u_n$. \\

\noindent By H\"older inequality
 \begin{equation} \label{thirdest}\|v_n\|_{L^2} \leq \|v_n\|_{L^4} \|\chi\|_{L^4}\stackrel{n\to\infty}\longrightarrow 0,\end{equation} which gives rise in view of the Sobolev embedding \eqref{2D-embedbis} to
  \begin{equation} \label{forest} \|v_n\|_{\wt{\mathcal
L}}\leq  \frac{1}{\sqrt{4\pi}}\|\nabla v_n\|_{L^2}+ \circ(1).\end{equation}
Taking advantage of the fact that the function $\chi$ takes its values in  $[0,1]$, we get
\begin{eqnarray*} \|\nabla v_n\|^2_{L^2} + \|\nabla \wt v_n\|^2_{L^2} &=& \| \chi \nabla u_n\|^2_{L^2}+ \| (1-\chi) \nabla u_n\|^2_{L^2}+ \circ(1) \\ &\leq& \|  \nabla u_n\|^2_{L^2}+ \circ(1).\end{eqnarray*}
Thus  by virtue of \eqref{bisest}, \eqref{thirdest} and \eqref{forest}
  $$ \|v_n\|_{{\mathcal
L}}=  \frac{1}{\sqrt{4\pi}}\|\nabla v_n\|_{L^2}+ \circ(1),$$
and $\|\nabla \wt v_n\|^2_{L^2}= \circ(1)$. \\

\noindent Now in view of the previous remark, we deduce that
$$ v_n(x)=\sqrt{\frac{\alpha_n}{2\pi}}\;{\mathbf L}\left(\frac{-\log|x|}{\alpha_n}\right)+{\rm
r}_n(x),\quad \|\nabla {\rm
r}_n\|_{L^2} \to 0 \quad \mbox{as} \quad n\to\infty,$$
which  ensures the explicit description \eqref{formprofiletild} since
 $$ v_n= u_n + \wt{\rm
r}_n, \quad \mbox{with} \quad \|\nabla \wt{\rm
r}_n\|_{L^2}\stackrel{n\to\infty}\longrightarrow 0.$$



\item It was shown in  \cite{H} that the sequence $( f_{\alpha_n})$  also writes under the form:
$$ f_{\alpha_n} (x)=  \wt {f}_{\alpha_n} (x) + \wt {\rm r}_n (x),$$
with \begin{equation} \label{logmoser}\wt {f}_{\alpha_n} (x) = \frac{1}{(2\pi)^2} \sqrt{\frac{2\pi}{\alpha_n}}\int_{ 1  \leq |\xi| \leq {\rm e}^{ \alpha_n}}  {\rm e}^{ i \, x \cdot \xi} \frac{1}{|\xi|^2}\, d \xi\,, \end{equation}
  and $\|
\wt {\rm r}_n\|_{H^1} \stackrel{n\to\infty}\longrightarrow 0$.
    \end{itemize}
    \end{rems}
    \medskip

\subsection{Linear Schr\"odinger equation}\label{seclin}
 It is well-known  that  the solutions of  the linear Schr\"odinger equation:
  \begin{equation}
\label{SL} \left\{
 \begin{aligned}
  & i\partial_t v + \Delta v=0, \\
  &v_{|t=0}=v_{0}\in H^1(\R^2),
 \end{aligned}
\right.
\end{equation} satisfy
 the conservation laws
\bq \label{consen} E_0(v,t) =  \|\nabla
v(t,\cdot)\|_{L^2}^2= \|\nabla
v(0,\cdot)\|_{L^2}^2 = E_0(v),\eq
\bq \label{consen2}   \|
v(t,\cdot)\|_{L^2}^2= \|
v(0,\cdot)\|_{L^2}^2\,,\eq
and for $t \neq 0$ the dispersive inequality
\begin{equation}
\label{disest}  \|
v(t,\cdot)\|_{L^\infty} \lesssim \frac {\|v(0, \cdot)\|_{L^1}} {|t|}\cdot \end{equation}
Combining \eqref{consen2}, \eqref{disest}  together with
the interpolation between $L^p$ spaces
imply that
\begin{equation}
\label{disestgen}
\forall t\in \R\setminus\{0\}\,,\ \forall p \in  [2, \infty]\,,\ \
\|v(t, \cdot)\|_{L^p} \lesssim \frac 1{|t|^{(1  -\frac 2 p
)}}\|v(0, \cdot)\|_{L^{p'}}.
\end{equation}
Thanks to   the so-called $TT^*$ argument which is the standard
method for converting the dispersive estimates  into inequalities involving suitable space-time Lebesgue norms of the
solutions, we get the following estimates known by Strichartz estimates  which will
be of constant use in this paper (see \cite{Caz1}):
\begin{prop}
\label{Stri}
Let $I\subset\R$ be a time slab, $t_0\in I$ and $(q,r)$, $(\tilde{q},\tilde{r})$ two $L^2$-admissible Strichartz pairs, {\rm i.e.},
\begin{equation}
\label{adpair}
2\leq r,\tilde{r}<\infty\quad\mbox{and}\quad \frac{1}{q}+\frac{1}{r}=\frac{1}{\tilde{q}}+\frac{1}{\tilde{r}}=\frac{1}{2}\;\cdot
\end{equation}
There exists a positive constant $C$ such that if $u$ is the solution of the Cauchy  problem
$$ \left\{
 \begin{aligned}
  &i\partial_t v+\Delta v = G(t,x), \\
  &v_{|t=t_0}=u_{0}\in H^1(\R^2),
 \end{aligned}
\right.
$$
then for $j \in \{0, 1\} $
\begin{equation}\label{str}
\|\nabla^j v\|_{L^q(I,L^r)}\leq C\Big( \|\nabla^jv(t_0)\|_{L^2}+\|\nabla^j G\|_{L^{{\tilde{q}}'}(I,L^{\tilde r'})}\Big),
\end{equation}
where $p'$  denotes  the conjugate exponent of $p$, defined by:
$$ \frac {1} p +  \frac  1 { p'} =1,  \;\; \mbox{with the rule that} \;\;\frac  1 {\infty} =0.\quad\quad $$
\end{prop}
\noindent Note in particular that $(q,r)=(4,4)$ is an admissible Strichartz pair and that (see \cite{BCD} for instance)
\begin{equation}
\label{embst}
W^{1,4}(\R^2)\hookrightarrow {\mathcal C}^{1/2}(\R^2).
\end{equation}Now, for any time slab $I\subset\R$, we shall denote
$$
\|v\|_{\mbox{\tiny ST}(I)}:=\sup_{j \in \{0, 1\}}\big(\|\nabla^jv\|_{L^4(I, L^4)} + \|\nabla^jv\|_{L^\infty(I, L^2)}\big)\,\mbox{and}\, \|v\|_{\mbox{\tiny ST}^*(I)}:=\sup_{j \in \{0, 1\}}\|\nabla^jv\|_{L^\frac4 3(I, L^\frac4 3)}.$$

\section{Virial identity   and scattering under smallness conditions}\label{Vid}
This section is devoted to the proof of basic estimates needed to develop the proof of Theorem \ref{Main-NLS}, namely virial identity and the $H^1$-scattering under smallness conditions. \subsection{Virial identity} \label{subVirial}
The aim of this paragraph is to present virial identity in the framework of the two dimensional   nonlinear Schr\"odinger equation
\eqref{NLS}.  For that purpose, let us introduce  a  smooth and radial function $\Phi$ satisfying $0\leq\Phi\leq1$, $\Phi(r)=r$, for all $r\leq1$, and $\Phi(r)=0$ for all $r\geq2$. For any positive real $R$ and any function   $u(x,t)$, we define
$$
V_R(t):=\int_{\R^2}\Phi_R(x)|u(x,t)|^2\;dx,
$$
where $\Phi_R(x):=R^2\Phi(\frac{|x|^2}{R^2})$. \\

As usual integrating by parts, we get in view of  \eqref{NLS}-\eqref{def f} and properties of $\Phi_R$ the following useful estimate known by virial identity:
\begin{lem} \label{lemvirial}
If $u$ is a solution to \eqref{NLS}-\eqref{def f}, then $V_R$ satisfies
\EQ{
\label{Virial}
\frac d{dt} V_R(t)=2{\mathcal I} \int_{\R^2}(\nabla\Phi_R (x)\cdot\nabla u (t,x))\,\bar u(t,x)\;dx,
}
where for $z \in \C$, ${\mathcal I}(z)$ denotes the imaginary part of $z$, and
\EQ{\label{virial}
\frac{d^2}{dt^2} V_R (t)&=8 \int_{\R^2}\Phi'\Big(\frac{|x|^2}{R^2}\Big)|\nabla u (t,x)|^2 \;dx+16\int_{\R^2}\Phi''\Big(\frac{|x|^2}{R^2}\Big) \frac{|x\cdot\nabla u(t,x)|^2 }{R^2} \;dx
\\ &-\int_{\R^2}|u(t,x)|^2\Delta^2\Phi_R(x)\;dx
+2\int_{\R^2}\Delta\Phi_R(x)\big(|u(t,x)|^2 \tilde f(|u(t,x)|^2)-g(|u(t,x)|^2)\big)\;dx,
}
where $\tilde f(s)={\rm e}^{4\pi s}-1-4\pi s$ and $\ds g(s)=\int_0^s\tilde f(\rho)\;d\rho$.
\end{lem}

\begin{proof}
According to the fact that $u \,  f( \bar u ) = \bar u \,  f(  u )$, we infer that if $u$ is a solution to \eqref{NLS}-\eqref{def f} then we have  $$ \frac d{dt} V_R (t) =i \, \int_{\R^2} \Phi_R(x) \big(\bar u \, \Delta u-u \, \Delta  \bar u\big)(t,x)\;dx.$$
This gives rise to \eqref{Virial} by integration by parts.

\medskip

\noindent Let us now go to the proof of  \eqref{virial}. As $u$ solves \eqref{NLS}, we deduce from  \eqref{Virial} that

\begin{eqnarray*}
\frac{d^2}{dt^2} V_R(t)&=&2\,{\mathcal I} \,\int_{\R^2}\nabla\Phi_R(x)\cdot\nabla(i\Delta u-if(u) )(t,x)\,\bar u (t,x)\;dx\\
&+&2\,{\mathcal I}\, \int_{\R^2}\nabla\Phi_R (x)\cdot\nabla u (t,x)\,(-i\Delta\bar u+if(\bar u))(t,x).
\end{eqnarray*}
Integrating by parts the first term of the right hand side of the above identity  gives
\begin{eqnarray*}
\frac{d^2}{dt^2} V_R(t)&=&2{\mathcal R} \int_{\R^2}\Delta\Phi_R (x) f(u(t,x))\bar u(t,x)+\big(f(u)\nabla\bar u + f(\bar u)\nabla u\big)(t,x) \cdot\nabla\Phi_R (x)\;dx\\
&-&2{\mathcal R}\int_{\R^2}\Delta\Phi_R(x)\bar u(t,x) \Delta u(t,x)+\nabla\Phi_R(x)\cdot \big(\nabla\bar u\Delta u+\nabla u\Delta\bar u\big)(t,x)\;dx, \end{eqnarray*} where for $z \in \C$, ${\mathcal R}(z)$ denotes the real part of $z$.

\smallskip
\noindent
Besides straightforward   computations lead to
\begin{eqnarray*}
J (t)&=& \int_{\R^2} \Delta\Phi_R(x) f(u(t,x))\bar u(t,x)+  \tilde f(|u (t,x)|^2)\nabla |u(t,x)|^2\cdot\nabla\Phi_R(x)\;dx\\
&=&\int_{\R^2}\Delta\Phi_R (x)\big(|u|^2\tilde f(|u|^2)- g(|u|^2)\big)(t,x)\;dx,
\end{eqnarray*}
where $$J (t):= \int_{\R^2}\Delta\Phi_R (x) f(u(t,x))\bar u(t,x)+\big(f(u)\nabla\bar u + f(\bar u)\nabla u\big)(t,x) \cdot\nabla\Phi_R (x)\;dx.$$
Along the same lines, we obtain
$$
2{\mathcal R}\int_{\R^2}\Delta\Phi_R(x)\bar u(t,x) \Delta u(t,x)=\int_{\R^2}|u(t,x)|^2\Delta^2\Phi_R(x)-2\Delta\Phi_R(x)|\nabla u(t,x)|^2,
$$
and
\begin{eqnarray*}
& & 4{\mathcal R}\int_{\R^2}\nabla\Phi_R(x)\cdot\nabla u (t,x)\Delta\bar u(t,x)\;dx= 2\int_{\R^2}\Delta\Phi_R(x)|\nabla u(t,x)|^2\;dx \\& & \qquad  \qquad \qquad - 8\int_{\R^2}\Phi'\Big(\frac{|x|^2}{R^2}\Big)
|\nabla u(t,x)|^2 dx - 16 \int_{\R^2}\Phi''\Big(\frac{|x|^2}{R^2}\Big)\frac{|x\cdot\nabla u (t,x)|^2}{R^2}\;dx.
\end{eqnarray*} This easily ensures the result.
\end{proof}

  \medbreak

As a  byproduct of virial identity, we get the following useful result:
\begin{cor}
\label{newesvir}  Let $u$ be  a solution to \eqref{NLS}-\eqref{def f} of mass $M$
and Hamiltonian $H \leq 1$. Then, there is a constant $C(M,H)$ such that for any positive real $\tau$, we have
 \begin{equation}
\label{eqvirGbis} \int^{t+\tau}_t \,\int_{|x| \leq 1} \,G(u(s,x))\, dx \, ds \leq C(M,H) \langle \tau \rangle,
\end{equation}
for all $t \in \R$, where $ G(u) = \left({\rm e}^{4\pi |u|^2}-1-4\pi |u|^2\right) |u|^2$.
\end{cor}
\begin{proof}
We shall proceed as in the proof of Lemma \ref{lemvirial} setting $$
V(t):=\int_{\R^2}\Phi_1(x)|u(x,t)|^2\;dx,
$$
where $\Phi_1(x):=\Phi(|x|^2)$, with $\Phi$ the function introduced above. Thus  thanks to \eqref{Virial}, we get
 \begin{equation}
\label{eqGV} \Big| \frac d{dt} V (t)  \Big|   \lesssim H^{\frac 1 2} \,M^{\frac 1 2}.\end{equation}
Moreover, in light of \eqref{virial}
\begin{eqnarray*}
\frac{d^2}{dt^2} V(t)&= & 8 \int_{\R^2}\Phi'(|x|^2)|\nabla u (t,x)|^2 \;dx +16\int_{\R^2}\Phi''(|x|^2) |x\cdot\nabla u(t,x)|^2  \;dx \\ &- &\int_{\R^2}|u(t,x)|^2\Delta^2\Phi_1(x)\;dx +2\int_{\R^2}\Delta\Phi_1(x)  \; \psi (u(t,x))\;dx,
\end{eqnarray*}
where, with the notations of Lemma \ref{lemvirial}, $\psi (u) := \big(|u|^2 \tilde f(|u|^2)-g(|u|^2)\big)$.\\

\noindent Straightforward computations lead to
$$ 2\int_{|x| \leq 1}\Delta\Phi_1(x)  \; \psi (u(t,x))\;dx =  8 \int_{|x| \leq 1} \psi (u(t,x))\;dx = \frac{d^2}{dt^2} V(t) + {\rm R}(t), $$
where the remainder term ${\rm R}(t)$ satisfies
 $$ \Big| {\rm R}(t) \Big|   \lesssim \|u(t,\cdot) \|^2_{L^2} + \|\nabla u(t,\cdot) \|^2_{L^2} + \int_{|x| \geq 1} \psi (u(t,x))\;dx.$$
 Besides by virtue of the radial estimate \eqref{Bound} and the conservation laws \eqref{mass}-\eqref{hamil}, we have
 $$ \| u(t,\cdot) \|_{L^\infty(\R, L^\infty (|x| \geq 1)} \lesssim C(M,H),$$
 which gives rise to
 $$\int_{|x| \geq 1} \psi (u(t,x))\;dx  \lesssim  \int_{|x| \geq 1} {\rm e}^{4\pi |u(t,x)|^2} |u(t,x)|^6 \;dx \lesssim C(M,H) \int_{|x| \geq 1}  |u(t,x)|^2 \;dx \lesssim C(M,H).$$
 This ensures that
 $$  {\rm R}(t) \leq C(M,H).$$
 Thus taking   into account of \eqref{eqGV},  we deduce that
 \begin{eqnarray*}  8 \int^{t+\tau}_t \,\int_{|x| \leq 1} \,\psi (u(s,x))\, dx \, ds &=& \int^{t+\tau}_t \, \frac{d^2}{ds^2} V(s) \, ds  + \int^{t+\tau}_t \,  {\rm R}(s)\, ds  \\  &\leq & \Big[\frac{d}{ds} V(s)\Big]^{t+\tau}_t + \tau \, C(M,H) \leq C(M,H) \langle \tau \rangle.\end{eqnarray*}
 This ends the proof of the Corollary under the fact that
$ G(u) \lesssim \psi(u)$.
\end{proof}

  \medbreak



\subsection{$H^1$-scattering under smallness conditions}
The purpose of this paragraph is to establish that  solutions to  \eqref{NLS}-\eqref{def f} scatter on $\R_\pm$, provided that  the evolution of the Cauchy data under the flow of  the linear Schr\"odinger equation is sufficiently small  in $L^\infty(\R_\pm,L^4)$ and strictly less than $\frac 1 {\sqrt{4\pi}}$ in $L^\infty(\R_\pm,\wt{\mathcal L}) $.
More precisely, we have the following lemma which turns out to be essential in our strategy:
\begin{lem}
\label{lem:1}  For all $M> 0$ and all $\delta > 0$, there is $\beta_0 = \beta_0 (M,\delta) > 0$  such that if $u$ is the solution to the Cauchy problem
$$ \left\{
 \begin{aligned}
  &i\partial_t u+\Delta u = f(u), \\
  &u_{|t=0}=u_{0}\in H^1(\R^2),
 \end{aligned}
\right.
$$
with $f(u)$ given by  \eqref{def f}, and $u_{0}$ satisfying
\begin{eqnarray*} H( u_{0}) &\leq &  1,   \quad   \| u_{0}\|_{L^2} = M, \\
\|{\rm e}^{ i t \Delta} \, u_0\|_{L^\infty(\R_\pm,\wt{\mathcal L})} &\leq&  \frac 1 {\sqrt{4\pi (1+\delta)}} \quad \mbox{and} \quad  \|{\rm e}^{ i t \Delta} \, u_0\|_{L^\infty(\R_\pm,L^4)} \leq \beta_0, \end{eqnarray*}
then $$\|u\|_{L^4(\R_\pm, W^{1,4})} <  \infty  \quad \mbox{and}  \quad \|f(u)\| _{\mbox{\tiny ST}^*(\R_\pm)}<  \infty.$$
Moreover    there is a positive constant $C$  such that \footnote{In fact, one can prove that  for all $0< r  <1$  there is a positive constant $C_r$  such that
$$ \|u(t,\cdot)-{\rm e}^{ i t \Delta} \, u_0\| _{\mbox{\tiny ST}(\R_\pm)} \leq   C_r  \,\beta_0^{1+r}.$$ But we  fixed $r=\frac 1  2$  in \eqref{estboot}  to avoid heaviness.}
\begin{equation} \label{estboot} \|u(t,\cdot)-{\rm e}^{ i t \Delta} \, u_0\| _{\mbox{\tiny ST}(\R_\pm)} \leq   C  \,\beta_0^{\frac 3 2}.\end{equation}\end{lem}
\begin{rem}
Note that \eqref{estboot} implies in a standard way the existence of $v_{\pm}$ in  $H_{rad}^1(\R^2)$ such that
$$
\|u (t, \cdot) - {\rm e}^{ i t \Delta} \,v_{\pm}\|_{H^1}\stackrel{t\to\pm\infty}\longrightarrow 0.
$$
\end{rem}
\begin{proof}
 The proof of Lemma \ref{lem:1} is based on the following bootstrap result.
\begin{lem}
\label{lem:boot}
Under the assumptions of Lemma \ref{lem:1},   there exist $\beta_0 = \beta_0 (M,\delta) > 0$ and $C = C (M,\delta) > 0$ such that for any $\beta \leq \beta_0$ and any real $T $,   if
\begin{equation} \label{estboot1} \|u(t,\cdot)-{\rm e}^{ i t \Delta} \, u_0\| _{\mbox{\tiny ST}( I_T)} \leq   \beta,\end{equation}
then\begin{equation} \label{estboot2} \|u(t,\cdot)-{\rm e}^{ i t \Delta} \, u_0\| _{\mbox{\tiny ST}(I_T)} \leq   C  \,\beta^{\frac 3 2},\end{equation}
where $I_T:= [0,T]$ if $T \geq 0$ and $I_T:= [T,0]$ if $T \leq 0$.
\end{lem}
\begin{proof} [Proof of Lemma \ref{lem:boot}]
Let us denote $ \|{\rm e}^{ i t \Delta} \, u_0\|_{L^\infty(\R_\pm,L^4)} =   \beta$.  In view of the  triangle inequality, Estimate  \eqref{estboot1}  ensures that
\begin{eqnarray} \label{bootsr1} \|u\|_{L^4(I_T, W^{1,4})} &\leq&  C(M) + \beta \quad \mbox{and} \quad  \\ \|u\|_{L^\infty(I_T, L^4)} &\leq& \beta + C \|u(t,\cdot)-{\rm e}^{ i t \Delta} \, u_0\| _{L^\infty(I_T,H^1)} \leq C \beta,   \label{bootsr2}\end{eqnarray}
which by virtue of  the  triangle and Sobolev inequalities leads to
\begin{equation} \label{firstbs} \|u\|_{L^\infty(I_T,\wt{\mathcal L})} \leq  \frac 1 {\sqrt{4\pi (1+\delta)}} + C \beta \leq  \frac{\Theta}{\sqrt{4\pi}}\, , \end{equation}
with $ \Theta < 1$   provided that $\beta$ is sufficiently small. \\

\noindent Let us also notice that the conservation laws \eqref{mass}-\eqref{hamil} imply that
$$  \| \nabla u(t,\cdot)\|_{L^2} \leq 1 \quad \mbox{and} \quad \| u(t,\cdot)\|_{L^2} = M, $$
and as $u$ solves the Cauchy problem \eqref{NLS}-\eqref{def f}, we have
$$ u(t,\cdot) = {\rm e}^{ i t \Delta} \, u_0 -i \int^t_{0} {\rm e}^{ i (t-s) \, \Delta}\, f(u(s,\cdot))\, ds.$$ Thus thanks to   Strichartz estimate \eqref{str}, we get
\begin{equation} \label{apst}  \|u(t,\cdot)-{\rm e}^{ i t \Delta} \, u_0\| _{\mbox{\tiny ST}(I_T)} \lesssim \|f(u)\| _{\mbox{\tiny ST}^*(I_T)}.\end{equation}
But  on the one hand $|f(u)| \lesssim  |u|^5 {\rm e}^{4 \pi
|u|^2}$,   thus for any   $0< \epsilon < 1$
\begin{eqnarray*}\|f(u(t,\cdot))\|^{\frac 4 3}_{ L^{\frac 4 3}(\R^2)} &\lesssim &\int_{\R^2}\,{\rm e}^{\frac {16 \pi} 3
|u(t,x)|^2} |u(t,x)|^{\frac {20} 3} \, dx \\ &\lesssim & {\rm e}^{4 \pi (1- \epsilon )\|u(t,\cdot)\|_{L^\infty}^2}  \int_{\R^2}\,{\rm e}^{4 \pi  (\frac 13 + \epsilon )
|u(t,x) |^2} |u(t,x)|^{\frac {20} 3} \, dx. \end{eqnarray*}
Choosing $\ds \epsilon < \frac 2 3$, we obtain in view of  Proposition \ref{usconsbis}
\begin{eqnarray}\nonumber \|f(u(t,\cdot))\|^{\frac 4 3}_{ L^{\frac 4 3}(\R^2)} &\lesssim&  {\rm e}^{4 \pi (1- \epsilon )\|u(t,\cdot)\|_{L^\infty}^2}  \int_{\R^2}\, |u(t,x)|^{\frac {20} 3} \, dx  \\ &\lesssim &  {\rm e}^{4 \pi (1- \epsilon )\|u(t,\cdot)\|_{L^\infty}^2} \|u(t,\cdot)\|_{L^\infty}^{\frac 8 3} \,\|u(t,\cdot)\|^4_{ L^4(\R^2)}\label{estsmall1},\end{eqnarray}
which leads to
\begin{equation} \label{firster} \|f(u(t,\cdot))\|^{\frac 4 3}_{ L^{\frac 4 3}(\R^2)} \lesssim \|u(t,\cdot)\|^{4}_{W^{ 1, 4 }(\R^2)} \,\|u(t,\cdot)\|^{\frac 8 3}_{ L^4(\R^2)}.\end{equation}
Indeed in the case when $\|u(t,\cdot)\|_{L^\infty} \leq 1$,   the estimate  \eqref{estsmall1} writes
$$ \|f(u(t,\cdot))\|^{\frac 4 3}_{ L^{\frac 4 3}(\R^2)} \lesssim  \|u(t,\cdot)\|_{L^\infty}^{\frac 8 3} \,\|u(t,\cdot)\|^4_{ L^4(\R^2)}, $$
which  thanks to \eqref{defS'h} and \eqref{embst} entails that
$$ \|f(u(t,\cdot))\|^{\frac 4 3}_{ L^{\frac 4 3}(\R^2)} \lesssim  \|u(t,\cdot)\|^{\frac 8 3}_{W^{ 1, 4 }(\R^2)} \, \|u(t,\cdot)\|^4_{ L^4(\R^2)}\lesssim \|u(t,\cdot)\|^{4}_{W^{ 1, 4 }(\R^2)} \,\|u(t,\cdot)\|^{\frac 8 3}_{ L^4(\R^2)}. $$
Furthermore in the case when $\|u(t,\cdot)\|_{L^\infty} \geq 1$, it follows  from \eqref{estsmall1} that
$$ \|f(u(t,\cdot))\|^{\frac 4 3}_{ L^{\frac 4 3}(\R^2)} \lesssim   {\rm e}^{4 \pi (1- \frac \epsilon 2)\|u(t,\cdot)\|_{L^\infty}^2}  \,\|u(t,\cdot)\|^4_{ L^4(\R^2)},$$
which in light of the logarithmic inequality (\ref{H-mu}) implies  for any fixed
 $\lambda>\frac{1}{\pi}$,
\begin{equation} \label{logap*}\|f(u(t,\cdot))\|^{\frac 4 3}_{L^{\frac 4 3}(\R^2)} \lesssim
 \|u(t,\cdot)\|^4_{ L^4(\R^2)} \left(1+\frac{\|u(t,\cdot)\|_{{\cC}^{\frac 1 2 }}}{\|u(t,\cdot)\|_{H_\mu}}\right)^{4 \pi (1- \frac \epsilon 2) \lambda
\|u(t,\cdot)\|_{H_\mu}^2}\cdot \end{equation}
Recalling that
$$ \|u(t,\cdot)\|_{H_\mu}^2 = \|\nabla u(t,\cdot)\|_{L^2}^2+\mu^2 \|u(t,\cdot)\|_{L^2}^2\leq 1+\mu^2 M^2,
$$
we infer  in view of \eqref{embst}
  that $\mu > 0$ and  $\lambda>\frac{1}{\pi}$ can be fixed so that
$$\|f(u(t,\cdot))\|^{\frac {4 } 3}_{ L^{\frac 4 3}(\R^2)} \lesssim   \|u(t,\cdot)\|^4_{ L^4(\R^2)} \Big(1+\|u(t,\cdot)\|^{4}_{W^{1,4}(\R^2)}\Big).$$
Since we are dealing  with the case $\|u(t,\cdot)\|_{L^\infty} \geq 1$,
this  ensures that
$$\|f(u(t,\cdot))\|^{\frac {4 } 3}_{ L^{\frac 4 3}(\R^2)} \lesssim   \|u(t,\cdot)\|^4_{ L^4(\R^2)} \|u(t,\cdot)\|^{4}_{W^{1,4}(\R^2)},$$
which ends the proof of  \eqref{firster} and leads by a time integration to
\begin{equation} \label{firstterm}  \|f(u)\|_{L^\frac4 3(I_T, L^\frac4 3)} \lesssim \, \|u\|^2_{L^\infty(I_T, L^4)} \|u\|^{3}_{L^4(I_T, W^{1,4})}.\end{equation}

\bigbreak

On the other hand observing that  $|\nabla f(u)| \lesssim  {\rm e}^{4 \pi |u|^2}\, |\nabla u|\, | u|^4$ and applying  H\"older inequality, we get
$$  \| \nabla f(u)\|^{\frac 4 3}_{ L^\frac4 3(I_T, L^\frac4 3)} \leq \| \nabla u\|^{\frac 4 3}_{ L^4(I_T, L^4)} \Big(\int_{\R \times \R^2}\,{\rm e}^{8 \pi
|u(t,x)|^2} |u(t,x)|^{8} \, dx \,dt \Big)^{\frac 2 3 }.$$ Arguing as above,
 we infer that if $\|u(t,\cdot)\|_{L^\infty} \leq 1$ then
\begin{eqnarray} \nonumber \int_{ \R^2}\,{\rm e}^{8 \pi
|u(t,x)|^2} |u(t,x)|^{8} \, dx &\lesssim & \int_{\R^2}\, |u(t,x)|^{8} \, dx  \\ &\lesssim &  \|u(t,\cdot)\|^4_{ L^4(\R^2)} \|u(t,\cdot)\|^{4}_{W^{1,4}(\R^2)}.\label{int1}\end{eqnarray}
 Moreover if $\|u(t,\cdot)\|_{L^\infty} \geq 1$,    we obtain for any $0< \epsilon < 1$
 $$ \int_{ \R^2}\,{\rm e}^{8 \pi
|u(t,x)|^2} |u(t,x)|^{8} \, dx \leq  {\rm e}^{4 \pi (1- \epsilon )\|u(t,\cdot)\|_{L^\infty}^2}  \int_{ \R^2}\,{\rm e}^{4 \pi (1+ \epsilon )
|u(t,x)|^2} |u(t,x)|^{8} \, dx.$$
 Taking advantage of Corollary \ref{uscons*}, we get
\begin{eqnarray*} \int_{ \R^2}\,{\rm e}^{8 \pi
|u(t,x)|^2} |u(t,x)|^{8} \, dx  &\lesssim &  {\rm e}^{4 \pi (1- \epsilon )\|u(t,\cdot)\|_{L^\infty}^2}  \left( \int_{ \R^2} |u(t,x)|^{8} dx + \Big(\int_{ \R^2}  |u(t,x)|^{8p}  dx\Big)^{\frac 1 p}\right)  \\ &\lesssim &  {\rm e}^{4 \pi (1- \epsilon )\|u(t,\cdot)\|_{L^\infty}^2} \left(  \|u(t,\cdot)\|_{L^4}^{4} \|u(t,\cdot)\|^4_{ L^\infty} + \|u(t,\cdot)\|_{L^4}^{\frac 4 p} \|u(t,\cdot)\|^{8-\frac 4 p}_{ L^\infty}\right)\\ &\lesssim &  {\rm e}^{4 \pi (1- \frac \epsilon 2 )\|u(t,\cdot)\|_{L^\infty}^2} \|u(t,\cdot)\|_{L^4}^{\frac 4 p},\end{eqnarray*}
provided that $ \epsilon $ is chosen sufficiently small. \\

 Applying  the  above lines of reasoning, we get  by virtue of the logarithmic inequality (\ref{H-mu}) still again in the case when $\|u(t,\cdot)\|_{L^\infty} \geq 1$
\begin{equation} \label{int2}  \int_{ \R^2}\,{\rm e}^{8 \pi
|u(t,x)|^2} |u(t,x)|^{8} \, dx \lesssim   \|u(t,\cdot)\|^{4}_{W^{1,4}} \|u(t,\cdot)\|_{L^4}^{\frac 4 p}. \end{equation}
Combining \eqref{int1} and \eqref{int2} we find that for any $p > 1 $, with $ p-1$ sufficiently small
$$  \| \nabla f(u)\|_{L^\frac4 3(I_T, L^\frac4 3)} \lesssim \, \|u\|^{\frac 2 p}_{L^\infty(I_T, L^4)} \|u\|^{3}_{L^4(I_T, W^{1,4})}.$$
Therefore, there is a positive constant $C = C (M,\delta)$ so that
\begin{equation} \label{secterm}  \|f(u)\| _{\mbox{\tiny ST}^*(I_T)} \leq C  \,  \|u\|^{3}_{L^4(I_T, W^{1,4})}\, \,\|u\|^{\frac 3 2}_{L^\infty(I_T, L^4)}.\end{equation}
Thus taking into account of \eqref{apst}, we deduce that
$$  \|u(t,\cdot)-{\rm e}^{ i t \Delta} \, u_0\| _{\mbox{\tiny ST}(I_T)} \leq C  \,  \|u\|^{3}_{L^4(I_T, W^{1,4})}\,\|u\|^{\frac 3 2}_{L^\infty(I_T, L^4)}, $$
which  together with \eqref{bootsr1} and \eqref{bootsr2} lead to \eqref{estboot1},  provided that $\beta$  is sufficiently small. This ends the proof of  the lemma.
\end{proof}

\medskip

By the standard continuity arguments, we deduce from Lemma \ref{lem:boot} that for $\beta_0$  sufficiently small  $\|u\| _{\mbox{\tiny ST}(\R_\pm)}   \lesssim \beta_0$. This implies in light of \eqref{secterm} that
$$ \|f(u)\| _{\mbox{\tiny ST}^*(\R_\pm)} \leq C  \beta^{\frac 3 2}_0,$$
which achieves the proof of Lemma \ref{lem:1}.
\end{proof}

  \medbreak



\section{Proof of  Theorem  \ref{Main-NLS}} \label{sec:mainth}
In this section, we shall demonstrate  in the critical case that  any solution $u$ to the Cauchy problem \eqref{NLS}-\eqref{def f} belongs to $L^4(\R, W^{1,4}(\R^2))$.  By considering $\overline u$   the conjugate of $u$, one can reduce the proof of  \eqref{Stg} to $\R_+$.  As it is mentioned in Section \ref{schgen}, this is achieved  in three steps. A first step where for any positive real sequence $(t_n)_{n \geq 0}$   tending to $+ \infty$, the strong  convergence to zero of the sequence $({\rm e}^{ i t \Delta}\, u(t_n,\cdot))_{n \geq 0}$ in $L^\infty(\R_+, L^4( \R^2))$ is settled.
A second step where a lack of compactness at infinity is derived, and a third step where the proof of the result  is completed.

\subsection{First step: strong convergence to zero in $L^\infty(\R_+, L^4( \R^2))$}
\begin{prop} \label{propstep1} Let  $u$ be a solution to \eqref{NLS}-\eqref{def f} and $(t_n)_{n \geq 0}$  be a positive real sequence tending to $+ \infty$.
Then
\begin{equation}
\label{eq:st1} \| {\rm e}^{ i t \Delta}\, u(t_n,\cdot) \| _{L^\infty(\R+, L^p( \R^2))}\stackrel{n\to\infty} \longrightarrow  0,
\end{equation}
for all $2< p< \infty$.
 \end{prop}
    \begin{proof}
 Proposition \ref{propstep1} stems from the following lemmas that we admit for a while.
     \begin{lem} \label{lemet1}
  If $u$ is a solution to \eqref{NLS}-\eqref{def f} and $(t_n)_{n \geq 0}$  is a positive real sequence tending to $+ \infty$, then
\begin{equation}
\label{eq:et1}  u(t_n,\cdot) \stackrel{n\to\infty}\rightharpoonup 0,  \quad \mbox{in}  \quad H^1( \R^2).
\end{equation}
  \end{lem}
      \begin{lem} \label{lemet2}
  Let $u$ be a solution to \eqref{NLS}-\eqref{def f},  $(t_n)$ and $(\tau_n)$ two  positive real sequences tending to $+ \infty$, then
\begin{equation}
\label{eq:et2}   \| {\rm e}^{ i \tau_n \Delta}\, u(t_n,\cdot) \| _{ L^p( \R^2)}\stackrel{n\to\infty} \longrightarrow  0,
\end{equation}
for all $2< p< \infty$.
  \end{lem}
   Before going into the proof of these fundamental  results, let us show how they  lead to Property \eqref{eq:st1}. For that purpose, we shall proceed by contradiction  assuming that the sequence  $(t_n)_{n \geq 0 }$ admits a subsequence $(t_{n_k})_{k \geq 0}$ and that  there exist    a  positive real sequence $(\tau_k)_{k \geq 0}$   and a positive real $\alpha_0$ such that for all $k \in \N$
  $$ \| {\rm e}^{ i \tau_k \Delta}\, u(t_{n_k},\cdot) \| _{ L^p( \R^2)} \geq \alpha_0.$$
   \noindent There are two possibilities up to extraction
\begin{enumerate}
\item $ \ds \tau_k \stackrel{k\to\infty}\longrightarrow   \tau_0 < \infty$  or else
\item $\ds \tau_k \stackrel{k\to\infty}\longrightarrow   + \infty$.\end{enumerate}
In the first case, the continuity of the flow implies that
$$  \| {\rm e}^{ i \tau_0 \Delta}\, u(t_{n_k},\cdot) \| _{ L^p( \R^2)} \gtrsim  \alpha_0,$$ which contradicts Lemma \ref{lemet1}  in light of Lemma \ref{compacity}, and the second case can not occur in view of Lemma  \ref{lemet2}.    \\

Now the heart of the matter consists to establish Lemmas \ref{lemet1} and \ref{lemet2}. Let us begin  by demonstrating the first lemma.

   \medskip

   \subsubsection*{Proof of Lemma \ref{lemet1}}
   To go to the proof of the result,  we shall proceed   by contradiction.
   Assume that there is a subsequence $(t_{n_k})_{k \geq 0}$ such that
   $$ u(t_{n_k},\cdot) \stackrel{k\to\infty}\rightharpoonup \varphi, \quad \mbox{in}  \quad H^1( \R^2),$$
  with  $\varphi \neq 0$ and for the shake of simplicity   denote $(t_{n_k})$ by $(t_n)$. This allows us to  write
   $$ u(t_n,\cdot) = \varphi + v_n, $$
   with $v_n \stackrel{n\to\infty}\rightharpoonup 0$ in $H^1( \R^2)$, which implies that
   $$ \|\nabla v_n \|^2 _{L^2} = \|\nabla u(t_{n},\cdot) \|^2 _{L^2} - \|\nabla \varphi \|^2 _{L^2} + \circ(1),\quad n\to\infty.  $$
   This  ensures the existence of a positive constant $\delta$ such that $\|\nabla v_n \| _{L^2} < 1 - 10\,  \delta$.
  By virtue of Lemma \ref{compacity}, one has
   \begin{equation}
\label{eq:v_nbis} \sup_{t\in [-1,1]} \| {\rm e}^{ i t \Delta}\, v_n\| _{L^p( \R^2)}\stackrel{n\to\infty} \longrightarrow  0, \end{equation}
for all $2< p< \infty$, which in view of Corollary \ref{cormoser} gives rise to
$$ \sup_{t\in [-1,1]} \| {\rm e}^{ i t \Delta}\, v_n\| _{\wt{\mathcal L}}  \leq \frac 1 {\sqrt{4\pi }}\|\nabla v_n \| _{L^2} + \circ(1).$$
 Thus for $n$ sufficiently large
 \begin{equation}
\label{eq:v_n} \sup_{t\in [-1,1]} \| {\rm e}^{ i t \Delta}\, v_n\| _{\wt{\mathcal L}}  \leq \frac {1-8 \delta} {\sqrt{4\pi }}\cdot  \end{equation}
 Besides, by density arguments one can  decompose $\varphi$ as follows:
 $$ \varphi:= \varphi_0 +  \varphi_1,$$
 where $\varphi_0 \in \cD$ and   $\| \varphi_1 \|_{H^1} \leq \delta$, which implies according to  Sobolev embedding \eqref{lackor2} that
 \begin{equation}
\label{eq:fi}\sup_{t\in [-1,1]} \| {\rm e}^{ i t \Delta}\, \varphi_1\| _{\wt{\mathcal L}}  \leq \frac { \delta} {\sqrt{4\pi }}\cdot \end{equation}
  Now our aim is to prove the existence of  $T(\delta) >0$, $\alpha(\delta) >0$ and $c(\delta) >0$  such that for any $n$ large enough,  we have $$ u(t,\cdot) = {\rm e}^{ i (t-t_{n}) \Delta}\, u(t_{n},\cdot)+  {\rm r}_n(t,\cdot),$$
   with  \begin{equation}\label{eqslab}  \|{\rm r}_n\|_{\mbox{\tiny ST}( [t_{n},t_{n}+T])} \leq c(\delta)  \, T^{\alpha(\delta)},  \end{equation}
   for any $0 \leq T \leq T(\delta) $.  \\

     \medskip
     We shall prove Claim \eqref{eqslab} by bootstrap argument, assuming  that for some  $T > 0$
     \begin{equation}\label{assum}  \|{\rm r}_n\|_{\mbox{\tiny ST}( I^n_T)} \leq \delta,  \end{equation}
     where $I^n_T:= [t_{n},t_{n}+T]$.  Actually under  Assumption \eqref{eqslab},   we have on the one hand
    \begin{equation} \label{bootstrap} \|u\|_{\mbox{\tiny ST}(I^n_T)}  \leq  \|u(t_n,.)\|_{H^1(\R^2)} +  \|{\rm r}_n\|_{\mbox{\tiny ST}(I^n_T)} \leq C(M,H)+ \|{\rm r}_n\|_{\mbox{\tiny ST}(I^n_T)}\leq C(M,H,\delta), \end{equation}
  and on the other hand
  we  may  decompose $u$ as follows
 $$ u(t,\cdot) = u_0(t,\cdot) + \wt u(t,\cdot),$$
 where $ \ds u_0(t,\cdot) =  {\rm e}^{ i (t-t_{n}) \Delta}\,\varphi_0$ and
  \begin{equation} \label{tildpetit} \|\wt u\| _{L^\infty(I^n_T,\wt{\mathcal L})}  \leq \frac { 1- 6 \delta} {\sqrt{4\pi }}\cdot  \end{equation}
  Indeed, by hypothesis
  $$ \wt u(t,\cdot)  =  {\rm e}^{ i (t-t_{n}) \Delta}\, (\varphi_1+ v_n)+ {\rm r}_n(t,\cdot), $$
which  clearly ensures the result thanks to \eqref{eq:v_n},  \eqref{eq:fi} and \eqref{eqslab} provided that $T  \leq 1$. \\

 \noindent   In order to establish \eqref{eqslab}, let us recall that since $u$ solves the Cauchy problem
$$ \left\{
 \begin{aligned}
  &i\partial_t u+\Delta u = f(u), \\
  &u_{|t=t_{n}}=u(t_{n},\cdot),
 \end{aligned}
\right.$$
the remainder term ${\rm r}_n$ writes
$${\rm r}_n(t,\cdot) = -i \int^t_{t_{n}} {\rm e}^{ i (t-s) \, \Delta}\, f(u(s,\cdot))\, ds.  $$
Thus taking advantage of  Strichartz estimates \eqref{str}, we deduce  that
\begin{equation}\label{basic}\|{\rm r}_n\|_{\mbox{\tiny ST}(I^n_T)} \lesssim  \|f(u)\|_{L^{\frac 4 3}(I^n_T, L^{\frac 4 3})} + \|\nabla f(u)\|_{L^{\frac 4 3}(I^n_T, L^{\frac 4 3})}\cdot\end{equation}
Firstly using the fact that $|f(u)| \lesssim  |u|^5 {\rm e}^{4 \pi
|u|^2}$,   we obtain
\begin{eqnarray*}\|f(u(t,\cdot))\|^{\frac 4 3}_{ L^{\frac 4 3}(\R^2)} &\lesssim &\int_{\R^2}\,{\rm e}^{\frac {16 \pi} 3
|u(t,x)|^2} |u(t,x)|^{\frac {20} 3} \, dx \\ &\lesssim & {\rm e}^{\frac {4 \pi} 3 (1- \frac {\delta} {10} )\|u(t,\cdot)\|_{L^\infty}^2}  \int_{\R^2}\,{\rm e}^{4 \pi  (1+ \frac {\delta} {10} )
|u(t,x)|^2} |u(t,x)|^{\frac {20} 3} \, dx \\ &\lesssim & {\rm e}^{\frac {4 \pi} 3 (1- \frac {\delta} {10} )\|u(t,\cdot)\|_{L^\infty}^2}  \int_{\R^2}\,{\rm e}^{4 \pi  (1+ \frac {\delta} {10} )
|u_0(t,x) + \wt u(t,x)|^2} |u(t,x)|^{\frac {20} 3} \, dx. \end{eqnarray*}
Remembering that $\varphi_0 \in \cD$ and thus $\| u_0\| _{L^\infty(\R,L^\infty(\R^2)} \lesssim \| \varphi_0\| _{H^2}$,  we infer that
\begin{eqnarray*}\|f(u(t,\cdot))\|^{\frac 4 3}_{ L^{\frac 4 3}(\R^2)} &\lesssim & {\rm e}^{\frac {4 \pi} 3 (1- \frac {\delta} {10} )\|u(t,\cdot)\|_{L^\infty}^2} {\rm e}^{c_\delta\|u_0\|_{L^\infty}^2} \int_{\R^2}\,{\rm e}^{4 \pi  (1+ \frac {2 \delta} {10} )
| \wt u(t,x)|^2} |u(t,x)|^{\frac {20} 3} \, dx\\ &\lesssim & {\rm e}^{\frac {4 \pi} 3 (1- \frac {\delta} {10} )\|u(t,\cdot)\|_{L^\infty}^2}  \int_{\R^2}\,{\rm e}^{4 \pi  (1+ \frac {2 \delta} {10} )
| \wt u(t,x)|^2} |u(t,x)|^{\frac {20} 3} \, dx.\end{eqnarray*}
Actually,  in  view of the continuous embedding of $H^1(\R^2)$ into $L^{\frac {20} 3}(\R^2)$ and the conservation laws \eqref{mass}-\eqref{hamil}, we deduce  that
$$ \int_{|\wt u| \leq 1}\,{\rm e}^{4 \pi  (1+ \frac {2 \delta} {10} )
| \wt u(t,x)|^2} |u(t,x)|^{\frac {20} 3} \, dx \lesssim \int_{\R^2}\, |u(t,x)|^{\frac {20} 3} \, dx \lesssim C(M,H).$$
Besides   $\delta$ being fixed small enough, we get  by virtue of \eqref{tildpetit}
\begin{eqnarray*} \int_{|\wt u| \geq 1}{\rm e}^{4 \pi  (1+ \frac {2 \delta} {10} )
| \wt u(t,x)|^2} |u(t,x)|^{\frac {20} 3} \, dx &\lesssim & \int_{\R^2}\big({\rm e}^{4 \pi  (1+ \frac {3 \delta} {10} )
| \wt u(t,x)|^2} -1 - 4 \pi  (1+ \frac {3 \delta} {10} )
| \wt u(t,x)|^2\big) dx  \\ &\lesssim & \kappa.\end{eqnarray*}
Therefore
$$ \|f(u(t,\cdot))\|^{\frac {4 } 3}_{ L^{\frac 4 3}(\R^2)} \lesssim  {\rm e}^{\frac {4 \pi} 3 (1- \frac {\delta} {10} )\|u(t,\cdot)\|_{L^\infty}^2}.$$
The logarithmic inequality (\ref{H-mu}) yields for any fixed
 $\lambda>\frac{1}{\pi}$,
\begin{equation}\label{logap1}  {\rm e}^{\frac {4 \pi} 3 (1- \frac {\delta} {10} )\|u(t,\cdot)\|_{L^\infty}^2}\lesssim
\left(1+\frac{\|u(t,\cdot)\|_{{\cC}^{\frac 1 2 }}}{\|u(t,\cdot)\|_{H_\mu}}\right)^{\frac {4 \pi} 3 (1- \frac {\delta} {10} )\lambda
\|u(t,\cdot)\|_{H_\mu}^2}. \end{equation}
Arguing as in the proof of Lemma  \ref{lem:1},
we infer that $\mu > 0$ and  $\lambda>\frac{1}{\pi}$ can be fixed so that
$$\|f(u(t,\cdot))\|^{\frac {4 } 3}_{ L^{\frac 4 3}(\R^2)} \lesssim  \Big(1+\|u(t,\cdot)\|^{\frac {4 } 3 (1- \frac {\delta^2} {100} )}_{W^{1,4}}\Big).$$
Along the same lines, using the fact that $|\nabla f(u)| \lesssim  {\rm e}^{4 \pi |u|^2}\, |\nabla u|\, | u|^4$, we obtain
\begin{eqnarray*} \| \nabla f(u(t,\cdot))\|^{\frac 4 3}_{ L^{\frac 4 3}(\R^2)} &\lesssim & \| \nabla u(t,\cdot)\|^{\frac 4 3}_{ L^{4}(\R^2)} \Big(\int_{\R^2}\,{\rm e}^{8 \pi
|u(t,x)|^2} |u(t,x)|^{8} \, dx \Big)^{\frac 2 3 } \\ &\lesssim & \| \nabla u(t,\cdot)\|^{\frac 4 3}_{ L^{4}(\R^2)} {\rm e}^{\frac {8 \pi} 3 (1- \frac {\delta} {10} )\|u(t,\cdot)\|_{L^\infty}^2} \Big(\int_{\R^2}\,{\rm e}^{4 \pi (1+ \frac {\delta} {10} )
|u(t,x)|^2} |u(t,x)|^{8} \, dx \Big)^{\frac 2 3 } \\ &\lesssim & \|  u(t,\cdot)\|^{\frac 4 3}_{ W^{1,4}} \Big(1+ \|u(t,\cdot)\|^{\frac 8 3(1- \frac {\delta^2} {100})}_{ W^{1,4}}\Big).\end{eqnarray*}
To summarize, we proved that for any fixed $\delta$ sufficiently small there is a positive constant $c_\delta$ so that
\begin{equation}\label{estnon*} \|  f(u(t,\cdot))\|^{\frac 4 3}_{ L^{\frac 4 3}(\R^2)} + \| \nabla f(u(t,\cdot))\|^{\frac 4 3}_{ L^{\frac 4 3}(\R^2)} \leq c_\delta \Big(1+ \|u(t,\cdot)\|^{4(1- \frac {\delta^2} {100})}_{ W^{1,4}}\Big),\end{equation}
which in view of \eqref{bootstrap} and \eqref{basic} gives rise to
\begin{equation}\label{estnon**}  \|{\rm r}_n\|_{\mbox{\tiny ST}(I^n_T)}  \leq c_\delta \big(T+ T^{ \frac {\delta^2} {100}} \, \|u\|^{4(1- \frac {\delta^2} {100})}_{\mbox{\tiny ST}(I^n_T)}\big)^{\frac 3 4} \lesssim c_\delta \,T^{ \alpha (\delta)},\end{equation}
with $\alpha (\delta)= \frac {3\delta^2} {400}\cdot$ By the standard continuity argument, this gives the required result for $T$ sufficiently small. \\

 \noindent Recall that we assumed that $\varphi \neq 0$, thus there is a positive constant $ \eta $ such that
$$\sup_{t\in [-1,1]} \| {\rm e}^{ i t \Delta}\, \varphi\| _{L^8}  \geq  \eta  > 0\cdot $$
Taking advantage of \eqref{eqslab}, there exists   $T > 0$ so that  for any $n$ large enough we have
$$\sup_{t\in I^n_T} \| {\rm r}_n(t,\cdot)\| _{L^8}  \leq  C  \sup_{t\in I^n_T} \| {\rm r}_n(t,\cdot)\| _{H^1} \leq  \frac  \eta 2\cdot$$
Consequently taking advantage of   \eqref{eq:v_nbis}, we deduce that  for any $t\in I^n_T$
$$  \| u(t,\cdot)\| _{L^8}  \geq  \| {\rm e}^{ i (t-t_{n}) \Delta}\, \varphi\| _{L^8} -  \| {\rm r}_n(t,\cdot)\| _{L^8}-  \| {\rm e}^{ i (t-t_{n}) \Delta}\, v_n\| _{L^8} \geq \eta -  \frac  \eta 2 +  \circ(1),$$
which ensures that
$$  \| u\| _{L^4(I^n_T,L^8)}  \geq   \frac  \eta 2  \, T^{\frac 1 4}+ \circ(1).$$
But in view of  a priori estimate \eqref{estpriori},  we necessarily have  $\| u\| _{L^4(I^n_T,L^8)}  = \circ(1)$ which yields a contradiction since $(u(t_n,\cdot))_{n  \in  \N}$ is bounded in $H^1(\R^2)$.  This ends the proof of the Lemma.

 \subsubsection*{Proof of Lemma \ref{lemet2}}
  According to Lemma \ref{compacity}, it suffices to prove that
  \begin{equation}
\label{eq:cv} {\rm e}^{ i \tau_n \Delta}\, u(t_n,\cdot) \stackrel{n\to\infty}\rightharpoonup 0,  \quad \mbox{in}  \quad H^1( \R^2).\end{equation}
  But since $( {\rm e}^{ i \tau_n \Delta}\, u(t_n,\cdot))$ is bounded in $H^1( \R^2)$, it is enough to prove \eqref{eq:cv}
in $L^2( \R^2)$. Thus by density arguments, we are reduced to prove that for any function $\varphi $ in $ \cS (\R^2)$  whose Fourier transform $\widehat{\varphi}$ belongs to  $\cD (\R^2 \setminus \{0\})$, we have
$$ \Big( {\rm e}^{ i \tau_n \Delta}\, u(t_n,\cdot), \varphi\Big)_{L^2( \R^2)} \stackrel{n\to\infty} \longrightarrow  0.$$
Since   $u$ solves the Cauchy problem \eqref{NLS}-\eqref{def f}, ${\rm e}^{ i \tau_n \Delta}\, u(t_n,\cdot)$ may be decomposed as follows:
  \begin{equation}
\label{eq:dec} {\rm e}^{ i \tau_n \Delta}\, u(t_n,\cdot)= v^{(0)}_n + v^{(1)}_n + v^{(2)}_n,\end{equation}
where
$$v^{(0)}_n:=   {\rm e}^{ i (\tau_n+t_n) \Delta}\,u_0,$$
and
$$v^{(1)}_n(x):=  -i \int_0^{t_{n}} {\rm e}^{ i (\tau_n+ t_n-s) \, \Delta}\, \Big[(1- \chi(x))f(u(s,x)) + 8 \pi^2\chi(x)|u (s,x)|^4 u (s,x)\Big]\, ds,$$
with $\chi$   a radial function in $\mathcal{D}(\R^2)$ valued in $[0,1]$ satisfying \begin{eqnarray*}
\chi(x)&=&\; \left\{
\begin{array}{cllll}1 \quad&\mbox{if}&\quad
|x|\leq \frac 1 2,\\ 0 \quad
&\mbox{if}&\quad |x|\geq 1.
\end{array}
\right.
\end{eqnarray*}
Obviously
$$ v^{(2)}_n(x)=  -i \int_0^{t_{n}} {\rm e}^{ i (\tau_n+ t_n-s) \, \Delta}\,  \chi(x) \, G_2( u (s,x)) \, ds,$$
where $G_2(u)= \Big({\rm e}^{4\pi |u|^2}-1-4\pi |u|^2-8\pi^2 |u|^4 \Big)\, u$.\\

 Now, we shall treat differently the three parts in \eqref{eq:dec}.  Firstly thanks to the dispersive estimate \eqref{disest}
$$ \|v^{(0)}_n \| _{ L^p( \R^2)} \stackrel{n\to\infty} \longrightarrow  0,
$$
for all $2< p< \infty$.\\

 Secondly,
$$ \|v^{(1)}_n \| _{ L^\infty( \R^2)} \stackrel{n\to\infty} \longrightarrow  0.
$$
Indeed, we know that
$$v^{(1)}_n(x):=  -i \int_0^{t_{n}} {\rm e}^{ i (\tau_n+ t_n-s) \, \Delta}\, G_1( u (s,x))\, ds,$$
where $$G_1( u (s,x))= \Big[(1- \chi(x))f(u(s,x)) + 8 \pi^2\chi(x)|u (s,x)|^4 u (s,x)\Big].$$
 Besides  the radial estimate \eqref{Bound} and the conservation laws \eqref{mass}-\eqref{hamil} ensure that the function $u$ is bounded away from the origin uniformly on $s\in \R$. Thus, there is a positive constant $C$ such that for any $(s,x) \in \R \times \R^2$, we have
$$ |G_1(u(s,x))| \leq C\, |u(s,x)|^5.  $$ Consequently
$$  \|G_1 (s,\cdot) \| _{ L^1( \R^2)} \lesssim \|u (s,\cdot) \|^5_{ L^5( \R^2)} \lesssim \|u (s,\cdot)\|_{ L^2( \R^2)}  \|u(s,\cdot) \|^4_{ L^8( \R^2)}\lesssim   \|u(s,\cdot) \|^4_{ L^8( \R^2)}$$
and then under the dispersive estimate \eqref{disest}
$$  \|v^{(1)}_n \| _{ L^\infty( \R^2)} \lesssim \int_0^{t_{n}} \frac 1 { (\tau_n+ t_n-s) }\, \|u(s,\cdot) \|^4_{ L^8( \R^2)}\, ds \lesssim \frac 1 { \tau_n} \,  \|u \|^4_{ L^4(\R,L^8( \R^2))}$$ which provides the result in view of the a priori estimate \eqref{estpriori}. \\

Finally, for any $\varphi $ in $ \cS (\R^2)$  whose Fourier transform $\widehat{\varphi}$ belongs to  $\cD (\R^2 \setminus \{0\})$, we infer that
$$ \Big( v^{(2)}_n, \varphi\Big)_{L^2( \R^2)} \stackrel{n\to\infty} \longrightarrow  0.$$
Indeed, by definition
$$ \Big( v^{(2)}_n, \varphi\Big)_{L^2( \R^2)}= i \int_0^{t_{n}}\Big( \chi \, G_2(u(s,\cdot)),  {\rm e}^{- i (\tau_n+ t_n-s) \, \Delta}\,  \varphi  \Big)_{L^2( \R^2)} \, ds.$$
But by repeated integration by parts, we obtain for any nonnegative integer $k$
 \begin{equation}
\label{eq:usus}\Big|\big({\rm e}^{- i (\tau_n+ t_n-s) \, \Delta}\,  \varphi\big)(s,x) \Big| \leq C_k  \Big| \frac {\langle x \rangle} {\tau_n+ t_n-s} \Big|^k, \end{equation}
which   in the particular case where $k=2$ implies that \begin{eqnarray*} \Big| \big( v^{(2)}_n, \varphi\big)_{L^2( \R^2)} \Big| &\lesssim&  \int_{\R^2} \int_0^{t_{n}} \frac {\langle x\rangle^2  \,  \chi(x) \, G_2(u(s,x))} {(\tau_n+ t_n-s)^2}  ds \, dx \\ &\lesssim& \int_0^{t_{n}} \frac {1} {(\tau_n+ t_n-s)^2}    \int_{|x| \leq 1} G_2(u(s,x))\,  ds \,dx  \\ &\lesssim&  \sum_{m=0}^{[t_n]} \frac {1} {(\tau_n+ t_n-(m+1))^2} \int_m^{m+1}   \int_{|x| \leq 1} G_2(u(s,x))\, ds \,  dx.\end{eqnarray*}
But in view of Corollary \ref{newesvir} we have  $$ \int_m^{m+1}    \int_{|x| \leq 1} G_2(u(s,x))\, ds \,dx   \lesssim \int_m^{m+1}   \int_{|x| \leq 1} G(u(s,x))\, ds \,dx\lesssim 1.$$ This implies that
$$\Big| \big( v^{(2)}_n, \varphi\big)_{L^2( \R^2)} \Big| \lesssim \sum_{m=0}^{[t_n]} \frac {1} {(\tau_n+ t_n-(m+1))^2} \lesssim \frac {1} {\tau_n} \stackrel{n\to\infty} \longrightarrow  0,$$
which achieves  the proof of the lemma.    \end{proof}
\subsection{Second step: lack of compactness at infinity}
The  approach  that we shall adopt to establish Theorem  \ref{Main-NLS}  relies  on virial identity and uses in a
crucial way the radial setting and particularly the fact that we
deal with bounded functions far away from the origin. Roughly speaking, virial identity asserts that any critical solution to  \eqref{NLS}-\eqref{def f}   displays necessarily a lack of compactness at infinity. More precisely, we have the following lemma:
 \begin{lem} \label{definfinil}
 Let $u$ be  a solution to the nonlinear Schr\"odinger equation \eqref{NLS}-\eqref{def f} satisfying
 \begin{equation}
\label{contcond} H(u_{0})= 1.\end{equation}
Then there exist  $\varepsilon_0 > 0$, $t_n \stackrel{n\to\infty} \longrightarrow  \infty $ and  $R_n \stackrel{n\to\infty} \longrightarrow \infty$  such that
 \begin{equation}
\label{mainpoint}
   \int_{|x|\geq R_n}\, |\nabla u (t_n,x)|^2 \, dx \geq \varepsilon_0.
\end{equation}
  \end{lem}
   \begin{proof}
   Clearly if $u$ is a solution of the concerned  nonlinear Schr\"odinger equation, then  in view of  \eqref{Virial} we have \begin{equation}\label{contradic1}
  \Big| \frac d{dt} V_R (t)  \Big|   \lesssim R,
\end{equation}
where $\ds V_R(t) = \int_{\R^2}\Phi_R(x)|u(t,x)|^2\;dx$, with $\Phi_R(x)=R^2\Phi(\frac{|x|^2}{R^2})$,  $\Phi$ being  a smooth and radial function satisfying $0\leq\Phi\leq1$, $\Phi(r)=r$, for all $r\leq1$, and $\Phi(r)=0$ for all $r\geq2$.  \\

 \noindent Besides, in light of \eqref{virial}
 $$ \frac{d^2}{dt^2} V_R(t) = 8 \int_{|x|\leq R}|\nabla u(t,x)|^2 \;dx+ 8 \int_{|x|\leq R}(|u|^2 \tilde f(|u|^2)-g(|u|^2))(t,x)\;dx+\wt{V_R}(t),  $$
 where as it was proved in Section \ref{subVirial}, $\tilde f(s)=e^{4\pi s}-1-4\pi s$ and $g(s)=\int_0^s\tilde f(\rho)\;d\rho$, and where the remainder term $\wt{V_R}(t)$ satisfies
 \begin{eqnarray*}\big|  \wt{V_R}(t) \big| &\lesssim&  \int_{R \leq |x|\leq 2 R}|\nabla u(t,x)|^2 \;dx+ \int_{R \leq |x|\leq 2 R}(|u|^2 \tilde f(|u|^2)-g(|u|^2))(t,x)\;dx\\ &+&\frac 1{R^2} \int_{R \leq |x|\leq 2 R}|u(t,x)|^2\;dx.\end{eqnarray*}
Now taking into account  the expressions of $\tilde f$ and $g$, we infer  in view of \eqref{BoundLP} that
\begin{eqnarray*}\int_{R \leq |x|\leq 2 R}(|u|^2 \tilde f(|u|^2)-g(|u|^2))(t,x)\;dx &\lesssim & \|u(t,\cdot) \|^6_{L^6( |x|\geq  R)} \\ &\lesssim & \|u(t,\cdot) \|^2_{L^2( |x|\geq  R)}\|u(t,\cdot) \|^4_{L^\infty( |x|\geq  R)} \\ &\lesssim & \frac 1 {R^2} \|u(t,\cdot) \|^4_{L^2( |x|\geq  R)} \|\nabla u(t,\cdot) \|^2_{L^2( |x|\geq  R)}.\end{eqnarray*}
Therefore, we deduce in light of the  conservation laws \eqref{mass}-\eqref{hamil} that
$$ \big|  \wt{V_R}(t) \big| \lesssim \|\nabla u(t,\cdot) \|^2_{L^2( |x|\geq  R)}+ \frac 1 {R^2}\cdot$$
Finally  the fact that
$$
\frac23g(|u|^2)\leq |u|^2\tilde f(|u|^2)- g(|u|^2),
$$
ensures the existence of  positive constants $C_1$ and $C_2$ such that
$$ \frac{d^2}{dt^2} V_R(t) \geq C_1 \int_{|x|\leq R} \big(|\nabla u(t,x)|^2 + F(u(t,x))\big) \;dx - C_2 \|\nabla u(t,\cdot) \|^2_{L^2( |x|\geq  R)} - \frac {C_2} {R^2}\cdot$$
This gives rise to
 \begin{eqnarray*} \frac{d^2}{dt^2} V_R(t) &\geq& C_1 H(u_0) - C_1\int_{|x|\geq R} (|\nabla u(t,x)|^2 + F(u(t,x))) dx - C_2 \|\nabla u(t,\cdot) \|^2_{L^2( |x|\geq  R)} - \frac {C_2} {R^2} \\ & \geq& C_1 - (C_1+C_2) \int_{|x|\geq R} |\nabla u(t,x)|^2 dx - \frac {C_2} {R^2} - C_1\int_{|x|\geq R} F(u(t,x)) dx. \end{eqnarray*}
Now  the solution $u$ is radial, so thanks to  \eqref{BoundLP} it is bounded on $\ds \big\{|x|\geq R \big\}$. Therefore
$$ \int_{|x|\geq R} F(u(t,x)) dx \lesssim \|u\|_{L^ \infty(\R, L^6( |x|\geq  R))}^6 \lesssim \|u\|_{L^\infty(\R, L^2( |x|\geq  R))}^2 \|u\|_{L^\infty(\R, L^\infty( |x|\geq  R))}^4.$$
Taking advantage of \eqref{BoundLP}, we deduce that
$$ \int_{|x|\geq R} F(u(t,x)) dx \lesssim \frac {1} {R^2}\|u\|_{L^\infty(\R, L^2( |x|\geq  R))}^4 \|\nabla u\|_{L^\infty(\R, L^2( |x|\geq  R))}^2, $$
which enables us due to the  conservation laws \eqref{mass}-\eqref{hamil} to infer that
$$ \frac{d^2}{dt^2} V_R(t) \geq  C_1 - (C_1+C_2) \int_{|x|\geq R} |\nabla u(t,x)|^2 dx - \frac {C_3} {R^2},$$
which easily ensures \eqref{mainpoint} taking into account of \eqref{contradic1}.  \end{proof}

  \medbreak

\subsection{Third step: study of the sequence $({\rm e}^{ i t \Delta} \, u(t_n, \cdot))$} In order to complete  the proof  of Theorem  \ref{Main-NLS}, we shall investigate the solution to the linear Schr\"odinger equation  $$ v_n:= {\rm e}^{ i t \Delta} \, u(t_n, \cdot),$$
where $(t_n)_{n \geq 0}$  is the sequence given by Lemma \ref{definfinil}.  In view of the first step,
the sequence  $(v_n)_{n \in \N}$  converges strongly to zero in $L^\infty(\R_+, L^4( \R^2))$.  For that purpose,  we have to distinguish two sub-cases depending on whether the sequence $v_n$ satisfies $$ \liminf_{n\to\infty}\;\|v_n\|_{L^\infty(\R_+,\wt{\mathcal L})} <  \frac 1 {\sqrt{4\pi}} \quad \mbox {or} \quad \lim_{n\to\infty}\;\|v_n\|_{L^\infty(\R_+,\wt{\mathcal L})} = \frac 1 {\sqrt{4\pi}}\cdot$$

 In the first situation where   we have $\liminf_{n\to\infty}\;\|v_n\|_{L^\infty(\R_+,\wt{\mathcal L})} < \frac 1{\sqrt{4\pi}}$,    Theorem  \ref{Main-NLS} derives immediately from Lemma \ref{lem:1}.  \\

 Let us at present  consider the more challenging  situation where  we are  dealing with  a sequence $(v_n)_{n \in \N}$ satisfying $\|v_n\|_{L^\infty(\R_+, L^4( \R^2))}\to 0$ and $\;\|v_n\|_{L^\infty(\R_+,\wt{\mathcal L})} \to  \frac 1 {\sqrt{4\pi}}\cdot$ This in particular means that  there exists a sequence $(\tau_n)_{n \in \N}$ of positive reals such that $$\|{\rm e}^{ i \tau_n \Delta} \, u(t_n, \cdot)\|_{\wt{\mathcal L}} \stackrel{n\to\infty}\longrightarrow \frac 1 {\sqrt{4\pi}}\cdot$$
Thus in view of \eqref{formprofiletild}
$$ {\rm e}^{ i \tau_n \Delta} \, u(t_n, \cdot)=  \sqrt{\frac{\alpha_n}{2\pi}}\;{\mathbf L}\left(\frac{-\log|\cdot|}{\alpha_n}\right)+{\rm
r}_n, $$
where $\alpha_n$ is a scale in the sense of Definition \ref{orthogen},  ${\mathbf L}$ is the Lions profile given by \eqref{lionsprof} and $\| \nabla{\rm
r}_n\|_{L^2} \stackrel{n\to\infty}\longrightarrow 0$.  \\

 \noindent There are two possibilities up to extraction
\begin{enumerate}
\item $ \ds \frac {\tau_n {\rm e}^{\alpha_n}} {\sqrt{\alpha_n}} \stackrel{n\to\infty}\longrightarrow   + \infty$  or else
\item $\ds \frac {\tau_n {\rm e}^{\alpha_n}} {\sqrt{\alpha_n}}  \  \lesssim 1$.\end{enumerate}
To investigate the case $(1)$, we shall decompose  $v_n$ as follows
$$ v_n(t,\cdot):= v^{(1)}_n(t,\cdot)+ v^{(2)}_n(t,\cdot)+ {\rm e}^{ i (t-\tau_n )\Delta} \, {\rm
r}_n,$$
where for $j \in \{1, 2\}$
$$ v^{(j)}_n(t,\cdot):= {\rm e}^{ i (t-\tau_n )\Delta} \, \sqrt{\frac{\alpha_n}{2\pi}}\;\psi^{(j)}\left(\frac{-\log|\cdot|}{\alpha_n}\right),$$
with, for fixed $0 < \delta <  \frac 1 2$,
\begin{eqnarray*}
\psi^{(1)}(s)&=&\; \left\{
\begin{array}{cllll}0 \quad&\mbox{if}&\quad
s\leq 0,\\ s \quad
&\mbox{if}&\quad 0\leq s\leq 1-\delta,\\
1-\delta \quad&\mbox{if}&s\geq 1-\delta,
\end{array}
\right.
\end{eqnarray*}
and
\begin{eqnarray*}
\psi^{(2)}(s)&=&\; \left\{
\begin{array}{cllll}0 \quad&\mbox{if}&\quad
s\leq 1-\delta,\\ s - (1-\delta) \quad
&\mbox{if}&\quad 1-\delta \leq s\leq 1,\\
\delta \quad&\mbox{if}&s\geq 1-\delta.
\end{array}
\right.
\end{eqnarray*}
Note that $\|v^{(i)}_n\|_{L^\infty(\R,L^2)}= \circ(1)$, for $i \in \{1, 2\}$. Moreover, by virtue of the Sobolev embedding \eqref{2D-embedbis} and the conservation laws \eqref{consen}-\eqref{consen2}
\begin{eqnarray}  \|v^{(1)}_n\|_{L^\infty(\R,\wt{\mathcal L})} &\leq& \frac 1 {\sqrt{4\pi}} \Big(\|\nabla v^{(1)}_n\|^2_{L^\infty(\R,L^2)}+ \| v^{(1)}_n\|^2_{L^\infty(\R,L^2)}\Big)^{\frac 12}
\nonumber \\ &\leq& \frac 1 {\sqrt{4\pi}} \Big(1-\delta+\circ(1) \Big)^{\frac 12} \leq  \frac 1 {\sqrt{4\pi}} \Big(1-\frac \delta 2 \Big)^{\frac 12}\label{behave1}, \end{eqnarray}
for $n$ large enough. \\

 \noindent  Besides if $\ds \frac {\tau_n {\rm e}^{\alpha_n}} {\sqrt{\alpha_n}} \to   + \infty$, then  in view of the dispersive estimate \eqref{disest}
$$  \sup_{t \in \R_ - }\|v^{(2)}_n(t,  \cdot)\|_{L^\infty} \lesssim \frac 1 {\tau_n}  \Big\|\sqrt{\frac{\alpha_n}{2\pi}}\;\psi^{(2)}\left(\frac{-\log|\cdot|}{\alpha_n}\right)\Big\|_{L^1}.
$$
But
$$ \Big\|\sqrt{\frac{\alpha_n}{2\pi}}\;\psi^{(2)}\left(\frac{-\log|\cdot|}{\alpha_n}\right)\Big\|_{L^1} \lesssim  \, \frac 1 {\sqrt{\alpha_n}}  {\rm e}^{-2 \alpha_n (1-\delta)}, $$which implies that
$$ \sup_{t \in \R_ - }\|v^{(2)}_n(t,  \cdot)\|_{L^\infty} \lesssim \, {\rm e}^{- \alpha_n (1- 2 \delta)} \stackrel{n\to\infty}\longrightarrow 0.$$
This gives rise to
\begin{equation}  \label{compinfty} \|v^{(2)}_n\|_{L^\infty(\R_ -, \wt{\mathcal L})} \stackrel{n\to\infty}\longrightarrow 0.\end{equation}
Finally knowing that $\|\nabla {\rm e}^{ i (t-\tau_n )\Delta} \, {\rm
r}_n\|_{L^\infty(\R, L^2)} \stackrel{n\to\infty}\longrightarrow 0$ and $\| {\rm e}^{ i (t-\tau_n )\Delta} \, {\rm
r}_n\|_{L^\infty(\R, L^2)} \lesssim 1$, we infer  in view of the refined estimate \eqref{apa4} that
$$ \| {\rm e}^{ i (t-\tau_n )\Delta} \, {\rm
r}_n\|_{L^\infty(\R, L^4)} \stackrel{n\to\infty}\longrightarrow 0,$$
and therefore by Corollary \ref{cormoser}
\begin{equation}  \label{compinftyrem} \| {\rm e}^{ i (t-\tau_n )\Delta} \, {\rm
r}_n\|_{L^\infty(\R, \wt{\mathcal L} )} \stackrel{n\to\infty}\longrightarrow 0.\end{equation}
Combining \eqref{behave1}, \eqref{compinfty} and \eqref{compinftyrem}, we deduce  that for any $n$ large enough
$$ \| {\rm e}^{ i t \Delta} \, u(t_n, \cdot)\|_{L^\infty(\R_-,\wt{\mathcal L})} <  \frac 1 {\sqrt{4\pi}} \Big(1-\frac \delta 4 \Big)^{\frac 12},$$  which  implies in view of Lemma \ref{lem:1} that  there is a positive constant C such that
$$ \|u( \cdot +t_n, \cdot)\|_{\mbox{\tiny ST}(\R_-) } \leq C  \quad  \mbox{and} \quad \|f(u( \cdot +t_n, \cdot))\| _{\mbox{\tiny ST}^*(\R_-)} \leq C.$$  Since $t_n \stackrel{n\to\infty}\longrightarrow + \infty$, this means that
$$ \|u\|_{\mbox{\tiny ST}(\R) }  +  \|f(u)\| _{\mbox{\tiny ST}^*(\R)} < \infty,$$
 which ends the proof of the result in that  case. \\

To handle the case $(2)$, we shall make use of the Fourier  approximation of the example  by Moser, namely  \eqref{logmoser} which gives rise to
\begin{equation} \label{lionsfourev}
 u(t_n, x) = \frac{1}{(2\pi)^2} \sqrt{\frac{2\pi}{\alpha_n}}\int_{ 1  \leq |\xi| \leq {\rm e}^{ \alpha_n}}  {\rm e}^{ i \, x \cdot \xi}  {\rm e}^{ i \, \tau_n  |\xi|^2}\frac{1}{|\xi|^2}\, d \xi +  \wt{\rm
r}_n(x),\end{equation}
with $\|\nabla \wt {\rm r}_n\|_{L^2} \stackrel{n\to\infty}\longrightarrow 0$.\\

 \noindent One can easily check that if  $\theta$ denotes  a regular function  valued in $[0,1]$ and satisfying
\begin{eqnarray*}
\theta(\rho)&=&\; \left\{
\begin{array}{cllll}0 \quad&\mbox{if}&\quad
\rho \leq 2,\\ 1 \quad
&\mbox{if}&\quad \rho \geq 3,
\end{array}
\right.
\end{eqnarray*}
then the above identity \eqref{lionsfourev} writes
\begin{eqnarray*}    u(t_n, x) &=& \frac{1}{(2\pi)^2} \sqrt{\frac{2\pi}{\alpha_n}}\int_{ 1  \leq |\xi| \leq {\rm e}^{ \alpha_n}}  {\rm e}^{ i \, x \cdot \xi}  {\rm e}^{  i \, \tau_n |\xi|^2}\, \theta( |\xi|) \, \theta(  {\rm e}^{ \alpha_n}-|\xi|)\, \frac{1}{|\xi|^2}\, d \xi \\ &+&   \, \wt{\rm
r}_n(x) + {\rm
r}^\#_n(x),\end{eqnarray*}
where
\begin{eqnarray*} {\rm
r}^\#_n (x) &=& \frac{1}{(2\pi)^2} \sqrt{\frac{2\pi}{\alpha_n}}\int_{ 1  \leq |\xi| \leq 3}  {\rm e}^{ i \, x \cdot \xi} \, {\rm e}^{ i \, \tau_n  |\xi|^2}\,\big(1- \theta( |\xi|)\big) \,  \frac{1}{|\xi|^2}\, d \xi \\ &+&  \frac{1}{(2\pi)^2} \sqrt{\frac{2\pi}{\alpha_n}}\int_{{\rm e}^{ \alpha_n} - 3  \leq |\xi| \leq {\rm e}^{ \alpha_n} }  {\rm e}^{ i \, x \cdot \xi} \, {\rm e}^{ i \, \tau_n  |\xi|^2}\, \big(1-\theta(  {\rm e}^{ \alpha_n}-|\xi|)\big)\, \frac{1}{|\xi|^2}\, d \xi,\end{eqnarray*}
and verifies $\|
{\rm
r}^\#_n\|_{H^1} \stackrel{n\to\infty}\longrightarrow 0$. \\

 \noindent Now let us denote
 $$ u_{0,n}(x):= \frac{1}{(2\pi)^2} \sqrt{\frac{2\pi}{\alpha_n}}\int_{ 1  \leq |\xi| \leq {\rm e}^{ \alpha_n}}  {\rm e}^{ i \, x \cdot \xi}  \, \theta( |\xi|) \, \theta(  {\rm e}^{ \alpha_n}-|\xi|)\, \frac{1}{|\xi|^2}\, d \xi.$$
 Using the classical estimate
 $$ \| |\cdot|\,{\rm e}^{ - i \tau_n \Delta} \, \nabla u_{0,n}\|_{L^2} \lesssim  \| |\cdot| \, \nabla u_{0,n}\|_{L^2} + |\tau_n|  \,\|  \Delta u_{0,n}\|_{L^2},$$
 we deduce that \begin{equation} \label{eqbound}\| |\cdot|\,{\rm e}^{ - i \tau_n \Delta} \, \nabla u_{0,n}\|_{L^2} \lesssim  1.\end{equation} Indeed, on the one hand
$$ \|  \Delta u_{0,n}\|^2_{L^2} \lesssim \frac{1}{\alpha_n} \int_{ 1  \leq |\xi| \leq {\rm e}^{ \alpha_n}}  d \xi \lesssim \frac{ {\rm e}^{2 \alpha_n}}{\alpha_n},  $$
thus $\tau_n  \,\|  \Delta u_{0,n}\|_{L^2} \lesssim 1$ since by assumption $ \frac {\tau_n {\rm e}^{\alpha_n}} {\sqrt{\alpha_n}}   \lesssim 1 \cdot$ \\

 \noindent  On the other hand \begin{eqnarray*} \| |\cdot| \, \nabla u_{0,n}\|_{L^2} &\leq & \| |\cdot| \, \nabla u_{0,n}\|_{L^2(|x|\leq 1)} + \| |\cdot| \, \nabla u_{0,n}\|_{L^2(|x|\geq 1)} \\ &\leq & \| \nabla u_{0,n}\|_{L^2} + \| |\cdot| \, \nabla u_{0,n}\|_{L^2(|x|\geq 1)}.\end{eqnarray*}
To achieve the proof of \eqref{eqbound}, it suffices then to bound $\| |\cdot| \, \nabla u_{0,n}\|_{L^2(|x|\geq 1)}$. To go to this end, we shall perform integration by parts with respect to the vector fields
$$  {\mathcal X} := -i  \sum^{i=2}_{i=1} \frac{ x_i \, \partial_{\xi_i}}{|x|^2},$$
which satisfies ${\mathcal X} ( {\rm e}^{ i \, x \cdot \xi})=  {\rm e}^{ i \, x \cdot \xi}$.   \\

 \noindent More precisely, taking advantage of the fact that $\ds \frac{\theta( |\xi|) \, \theta(  {\rm e}^{ \alpha_n}-|\xi|)}{|\xi|}$ is compactly supported, we get for any nonnegative integer $N$ and $i \in \{1,2\}$
 $$ ( \partial_ {x_i}  u_{0,n}) (x)=  \frac{(-1)^N}{(2\pi)^2} \sqrt{\frac{2\pi}{\alpha_n}}\int_{ 1  \leq |\xi| \leq {\rm e}^{ \alpha_n}}  {\rm e}^{ i \, x \cdot \xi}  \, {\mathcal X}^N \Big(\frac{i \,\xi_ {i}\,\theta( |\xi|) \, \theta(  {\rm e}^{ \alpha_n}-|\xi|)}{|\xi|^2}\Big)\, d \xi.$$
 Taking advantage of the fact that $$\theta( |\xi|) \, \theta(  {\rm e}^{ \alpha_n}-|\xi|)  \equiv 1 \quad \mbox{for} \quad 3 \leq |\xi| \leq  {\rm e}^{ \alpha_n} -3,$$  we obtain
 $$ \Big| {\mathcal X}^N \Big(\frac{i \,\xi_ {i}\,\theta( |\xi|) \, \theta(  {\rm e}^{ \alpha_n}-|\xi|)}{|\xi|^2}\Big)\Big|  \lesssim \frac{1}{|x|^N} \, \left(\frac{1}{|\xi|^{N+1}} +\frac{\bf{1}_{  \{1  \leq |\xi| \leq 3 \}\cup  \{ {\rm e}^{ \alpha_n} -3  \leq |\xi| \leq {\rm e}^{ \alpha_n} \} } }{|\xi|} \right),$$
 which leads  to
 $$ \big|\, |x| \, \nabla u_{0,n}(x)\big| \lesssim  \frac{1}{\sqrt{\alpha_n}\, |x|^{N-1}}  $$
uniformly for $|x|\geq 1$ and achieves the proof of the fact that
 $ \| |\cdot| \, \nabla u_{0,n}\|^2_{L^2} \lesssim 1. $\\

 In conclusion, we have
 $$ u(t_n, \cdot) = {\rm e}^{- i \tau_n \Delta} \, u_{0,n} + {\rm
r}_n^\flat,$$
with $\| |\cdot|\,{\rm e}^{ - i \tau_n \Delta} \, \nabla u_{0,n}\|_{L^2} \lesssim  1$ and $\| \nabla{\rm
r}^\flat_n\|_{L^2} \stackrel{n\to\infty}\longrightarrow 0$.  Thus, for all $R> 0$
$$   \left( \int_{|x|\geq R}\, |\nabla u (t_n,x)|^2 \, dx\right)^{\frac 1 2} \lesssim \frac 1 R \| |\cdot|\,{\rm e}^{ - i \tau_n \Delta} \, \nabla u_{0,n}\|_{L^2} + \| \nabla{\rm
r}^\flat_n\|_{L^2} \lesssim \frac 1 R + \circ(1).  $$
This implies that for $n$ and $R$ large enough
$$ \left( \int_{|x|\geq R}\, |\nabla u (t_n,x)|^2 \, dx\right)^{\frac 1 2} \leq \frac {\varepsilon_0} 2,$$
which contradicts \eqref{mainpoint} and shows that this case cannot occur.

  \medbreak




 \appendix





\section{ A useful  Moser-Trudinger inequality}\label{usefullMT}
\subsection{Notion of rearrangement}
Before proving  Proposition  \ref{usconsbis}  which plays a key role in our approach, let us  start with an overview of the notion of rearrangement of functions involving in its proof.  This notion consists in associating to any measurable function vanishing at infinity, a nonnegative decreasing radially symmetric function. To begin with, let us first define the symmetric rearrangement of a measurable set.
\begin{defi}
\label{rearset}
Let $A\subset\R^d$ be a Borel set of finite Lebesgue measure. We define $A^*$, the symmetric rearrangement of $A$, to be the open ball centered at the origin whose volume is that of $A$. Thus,
$$A^*=\{x\,:\, |x|<R\}\quad\mbox{with}\quad\left(|\mathbb{S}^{d-1}|/d\right)R^d=|A|\,,
$$
where $|\mathbb{S}^{d-1}|$ is the surface area of the unit sphere $\mathbb{S}^{d-1}$.

\end{defi}
This definition allows us to define in an obvious way the symmetric-decreasing rearrangement of a characteristic function of a set, namely
$$
\chi^*_{A}:=\chi_{A^*}\,.
$$
More generally, if $f :\R^d\to \C$ is a measurable function vanishing at infinity {\rm i.e.}
$$
\forall\;t>0,\qquad\; \Big|\{x\,:\,|f(x)|>t\}\Big|<\infty\,,
$$
then we define the symmetric decreasing rearrangement, $f^*$, of $f$ as
\bq
\label{Rear}
f^*(x)=\int_0^\infty\,\chi^*_{\{|f|>t\}}(x)\,dt\,.
\eq

In the following proposition,  we collect  without proofs  some  features of the rearrangement $f^*$ (for a
complete presentation and more details, we refer the reader to \cite{Kes, Analysis,  Talenti} and the references therein), namely that the process of Schwarz symmetrization minimize the energy and preserves Lebesgue and Orlicz norms:
\begin{prop}
\label{norminvar}
Let  $f\in H^1(\R^2)$ then
\beqn
\|\nabla f\|_{L^2}&\geq &\|\nabla f^*\|_{L^2},\\
\|f\|_{L^p}&=&\|f^*\|_{L^p},\\
\|f\|_{\cL}&=&\|f^*\|_{\cL}\,.
\eeqn
\end{prop}
\subsection{Proofs  of Proposition \ref{usconsbis} and Corollaries  \ref{mostrud},  \ref{cormoser} and \ref{uscons*}}
\subsubsection{Proof  of Proposition  \ref{usconsbis}} \label{proposition} The proof of  inequalities \refeq{Mosbis2}  uses in a crucial way the rearrangement of functions defined above.
 By virtue of  density arguments and Proposition \ref{norminvar},
one can reduce to the case of a nonnegative radially symmetric and nonincreasing function  $u$ belonging to  ${\cD}(\R^2)$.  With this choice, let us introduce the function $$w(t)=\sqrt{4\pi} \, u(|x|), \quad \mbox{where} \quad|x|={\rm e}^{-\frac{t}{2}}.$$ It is then obvious  that the functions  $ w(t) $ and   $ \dot{w}(t) $ are nonnegative and
\begin{eqnarray*}
  \int_{\mathbb{R}^2}|\nabla u(x)|^2 \,dx &=& \int_{-\infty}^{+\infty}|{w}'(t)|^2 \,dt, \\ \int_{\mathbb{R}^2}|u(x)|^p\,dx &=& \frac{\pi}{(\sqrt{4\pi} )^p}\int_{-\infty}^{+\infty} |w(t)|^p~{\rm e}^{-t}\,dt \quad \mbox{and}
   \\
 \int_{\mathbb{R}^2} {\rm e}^{\alpha |u(x)|^2} \, |u(x)|^p\, dx &=& \frac{\pi}{(\sqrt{4\pi} )^p} \int_{-\infty}^{+\infty} {\rm e}^{\frac{\alpha}{4\pi}|w(t)|^2}\, |w(t)|^p \,{\rm e}^{-t}\,dt.
\end{eqnarray*} Besides since $u \in {\cD}(\R^2)$,  there exists $t_0 \in \R$ such that
$$w(t) = 0, ~\mbox{\, for \, }t \leq t_0. $$ So we are reduced to prove that for all $\beta\in  [0,1 [$, there  exists $C_{\beta}\geq 0$ so that \begin{equation}\label{res}
\int_{-\infty}^{+\infty} {\rm e}^{\beta |w(t)|^2}  |w(t)|^p {\rm e}^{-t} dt \leq C_{\beta} \int_{-\infty}^{+\infty}|w(t)|^p {\rm e}^{-t}\,dt,\quad \forall \,\beta\in  [0,1 [,
\end{equation}
provided that
\begin{equation}\label{M3}
    \int_{-\infty}^{+\infty}|{w}'(t)|^2 dt \leq1.
\end{equation}
For that purpose, let us  set $$T_{0}:=\sup \left\{t \in \mathbb{R},~w(t)\leq 1\right\}.$$
Knowing that $w$ is nonnegative  and increasing function, we deduce that
$$ w:   ] -\infty,T_0 ] \longrightarrow  [0,1]. $$
It is then obvious that
\begin{equation*}
\int_{-\infty}^{T_0} {\rm e}^{\beta |w(t)|^2}  |w(t)|^p {\rm e}^{-t} dt \leq {\rm e}^{\beta} \int_{-\infty}^{T_0} |w(t)|^p {\rm e}^{-t} dt.
\end{equation*}
To estimate  the integral on  $[T_0,+\infty[$, let us first notice that   in view of \eqref{M3},  we have for all $t\geq T_0$  \begin{eqnarray*}
  w(t) &=& w(T_0)+\int_{T_0}^t{w}'(\tau)d\tau \\
  &\leq& w(T_0)+(t-T_0)^{\frac{1}{2}}\left(\int_{T_0}^{+\infty}{w'}(\tau)^2 d\tau\right)^{\frac{1}{2}}\\
  &\leq& 1+(t-T_0)^{\frac{1}{2}}.\end{eqnarray*}
Thus, using the fact  that for any $\varepsilon > 0 $ and any $s\geq 0$,  we have
$$(1+s^{\frac{1}{2}})^2 \leq (1+ \varepsilon) s + 1+\frac{1}{\varepsilon}=(1+ \varepsilon) s+C_{\varepsilon},$$
we infer that  for any $\varepsilon > 0 $ and all $t\geq T_0$
\begin{equation}\label{estus}|w(t)|^2 \leq  (1+\varepsilon)(t-T_0)+ C_{\varepsilon}. \end{equation}
Now  $\beta$ being fixed in  $[0,1 [$, let us choose    $\varepsilon>0$  so that $\beta(1+  \varepsilon)^2<1$.  Therefore  in light  of  \eqref{estus}, we obtain
\begin{eqnarray*}
        \int_{T_0}^{+\infty}  {\rm e}^{\beta |w(t)|^2}  |w(t)|^p {\rm e}^{-t} dt &\leq &  \int_{T_0}^{+\infty}{\rm e}^{\beta ((1+\varepsilon)(t-T_0) + C_{\varepsilon})}\,\big( (1+\varepsilon)(t-T_0)+ C_{\varepsilon}\big)^{\frac p 2}\, {\rm e}^{-t}dt \\
         &\leq&  C(p,\varepsilon)\int_{T_0}^{+\infty}{\rm e}^{\beta(1+\varepsilon) ((1+\varepsilon)(t-T_0) + C_{\varepsilon})}\, {\rm e}^{-t}dt \\
        &\leq& C(p,\varepsilon) \, \frac{{\rm e}^{\beta (1+\varepsilon) C_{\varepsilon}-T_0}}{1-\beta(1+\varepsilon)^2}\cdot
      \end{eqnarray*}
           Finally observing that
$${\rm e}^{-T_0}=\int_{T_0}^{+\infty} {\rm e}^{-t}\,dt \leq \int_{T_0}^{+\infty}  |w(t) |^p \, {\rm e}^{-t}\,dt,$$
we end up with the result.

\subsubsection{Proof   Corollary \ref{mostrud}}  Corollary \ref{mostrud} derives immediately from Proposition \ref{usconsbis} once we have observed that
$$ \left({\rm e}^{\alpha
|u(x)|^2}-1- \alpha
|u(x)|^2\right) \lesssim {\rm e}^{\alpha
|u(x)|^2} |u(x)|^4. $$
\subsubsection{Proof   Corollary  \ref{cormoser}}  Corollary \ref{cormoser} follows easily  from   Proposition \ref{mostrud} by applying the estimate \eqref{Mosbis} to the sequence $(v_n)$ defined by
$$v_n:= \frac{u_n} {\|\nabla u_n\|_{L^2} + \sqrt{\|u_n\|_{L^4}}}\cdot$$
\subsubsection{Proof   Corollary \ref{uscons*}} Firstly  it is obvious  that
$$ \int_{ |u| \leq 1}\,{\rm e}^{4\pi (1+\varepsilon)
|u(x)|^2} \, |u(x)|^p \, dx \lesssim \int_{  \R^2} \,|u(x)|^p \,dx.$$
Secondly  making use of H\"older inequality, we infer that for any real $r  \geq 1$, we have
\begin{eqnarray*}
      \int_{ |u| \geq 1}\, {\rm e}^{4\pi (1+\varepsilon)
|u(x)|^2}\, |u(x)|^{p} \, dx &\lesssim &   \Big(\int_{ |u| \geq 1}\, {\rm e}^{4\pi r (1+\varepsilon- \varepsilon(2+\frac 1 \delta))
|u(x)|^2}\, |u(x)|^{pr} \, dx\Big)^{\frac 1 r } \\
         &\times&  \Big(\int_{ |u| \geq 1}\, {\rm e}^{4\pi r'  \varepsilon(2+\frac 1 \delta)
|u(x)|^2}\, dx\Big)^{\frac 1 {r' }},
      \end{eqnarray*}
where $r'$  denotes  the conjugate exponent of $r$.\\

 \noindent The choice  $\ds r' = \frac {1+ \delta} {\varepsilon(2+\frac 1 \delta)}$ gives rise to
\begin{eqnarray*}
     \int_{ |u| \geq 1}\, {\rm e}^{4\pi r'  \varepsilon(2+\frac 1 \delta)
|u(x)|^2}\, dx &=  &    \int_{ |u| \geq 1}\, {\rm e}^{4\pi (1+ \delta)
|u(x)|^2}\, dx \\
         &\leq & C( \delta) \int_{ |u| \geq 1}\,  \left({\rm e}^{4\pi (1+ \delta)
|u(x)|^2} -1- 4\pi (1+ \delta)
|u(x)|^2\right)\, dx \\
         &\leq &  C( \delta) \kappa,
      \end{eqnarray*}
in view of  the assumption $\|u\|_{\wt {\mathcal
L}} \leq \frac 1 {\sqrt{4\pi (1+ 2 \delta)}}\cdot$ \\

 \noindent Consequently
 $$      \int_{ |u| \geq 1}\, {\rm e}^{4\pi (1+\varepsilon)
|u(x)|^2}\, |u(x)|^{p} \, dx \leq  (C( \delta) \kappa)^{\frac 1 {r' }}   \Big(\int_{ |u| \geq 1}\, {\rm e}^{4\pi r (1+\varepsilon- \varepsilon(2+\frac 1 \delta))
|u(x)|^2}\, |u(x)|^{pr} \, dx\Big)^{\frac 1 r }\cdot $$
But $$ r (1+\varepsilon- \varepsilon(2+\frac 1 \delta))=  1- \frac  {\delta  \varepsilon} {1 + \delta }+ {\mathcal{O}}(\varepsilon^2),$$
which  by virtue of  Proposition \ref{usconsbis} leads to
 $$      \int_{ |u| \geq 1}\, {\rm e}^{4\pi (1+\varepsilon)
|u(x)|^2}\, |u(x)|^{p} \, dx \leq  (C( \delta) \kappa)^{\frac 1 {r' }}   \Big(\int_{ |u| \geq 1}\, |u(x)|^{pr} \, dx\Big)^{\frac 1 r },$$
provided that $\varepsilon$ is sufficiently small. This ends the proof of the corollary.




\end{document}